\documentclass[10pt]{amsart}

\usepackage{amsmath, amssymb, amsthm, marginnote, ltxcmds, adjustbox, stmaryrd, graphicx}

\usepackage{mathtools, stackengine, changepage, ragged2e}

\usepackage{todonotes}

\definecolor{cite}{HTML}{11871E}
\definecolor{url}{HTML}{698996}
\definecolor{link}{HTML}{912F1B}

\usepackage[pdfencoding=unicode, colorlinks=true, linkcolor=link, citecolor=cite, urlcolor=url, linktocpage]{hyperref}

\usepackage[alphabetic, initials]{amsrefs}

\usepackage{tikz}
\usetikzlibrary{cd}
\usetikzlibrary{arrows, arrows.meta, positioning, calc}
\tikzcdset{arrow style=tikz, diagrams={>={Straight Barb[scale=0.8]}}}

\tikzstyle{arrow} = [-{Straight Barb[scale=0.8]}, line width=0.2mm]

\usepackage{enumitem}
\newenvironment{myenum}[1]{%
\begin{enumerate}[label=#1,topsep=1pt,itemsep=0pt,partopsep=1pt,parsep=1pt]%
}%
{\end{enumerate}%
}

\usepackage[
	letterpaper,
	twoside=false,
	textheight=22cm,
	textwidth=14cm,
	marginparsep=0.75cm,
	marginparwidth=2.5cm,
	heightrounded,
	centering
]{geometry}

\usepackage[protrusion=true,expansion=true]{microtype}

\usepackage[bitstream-charter]{mathdesign}
\usepackage[T1]{fontenc}
\usepackage{stmaryrd}

\usepackage{biolinum}

\renewcommand{\mathsf}[1]{\text{\normalfont\sffamily#1}}
\newcommand{\mathsfbf}[1]{\text{\normalfont\bf\sffamily#1}}

\DeclareMathAlphabet{\eur}{U}{zeus}{m}{n}
\renewcommand{\mathcal}[1]{\eur{#1}}

\usepackage[capitalise]{cleveref}

\Crefname{prop}{Proposition}{Propositions}
\Crefname{lem}{Lemma}{Lemmas}
\Crefname{cor}{Corollary}{Corollaries}
\Crefname{thm}{Theorem}{Theorems}

\Crefname{defn}{Definition}{Definitions}
\Crefname{notation}{Notation}{Notations}
\Crefname{conj}{Conjecture}{Conjectures}
\Crefname{ass}{Assumption}{Assumptions}
\Crefname{expt}{Expectation}{Expectations}

\Crefname{rmk}{Remark}{Remarks}
\Crefname{question}{Question}{Questions}
\Crefname{expl}{Example}{Examples}

\Crefname{figure}{Figure}{Figures}

\crefformat{equation}{(#2#1#3)}
\crefformat{section}{\S#2#1#3}
\crefmultiformat{section}{\S\S#2#1#3}{and~#2#1#3}{, #2#1#3}{, and~#2#1#3}

\theoremstyle{plain}
\newtheorem{prop}[subsubsection]{Proposition}
\newtheorem{lem}[subsubsection]{Lemma}
\newtheorem{cor}[subsubsection]{Corollary}
\newtheorem{thm}[subsubsection]{Theorem}
\newtheorem*{thm*}{Theorem}

\theoremstyle{definition}
\newtheorem{defn}[subsubsection]{Definition}

\newtheorem{expt}[subsubsection]{Expectation}

\theoremstyle{remark}
\newtheorem{rmk}[subsubsection]{Remark}
\newtheorem*{rmk*}{Remark}
\newtheorem{expl}[subsubsection]{Example}

\numberwithin{equation}{subsection}

\newcommand{\teq}{\addtocounter{subsubsection}{1}\tag{\thesubsubsection}}

\newcommand{\aff}{\mathsf{aff}}
\DeclareMathOperator{\Aff}{\mathsf{Aff}}
\newcommand{\aft}{\mathsf{aft}}
\newcommand{\auto}{\mathsf{auto}}

\DeclareMathOperator{\Bun}{\mathsf{Bun}}

\newcommand{\Cech}{\v{C}ech}
\newcommand{\coCech}{co\Cech}

\DeclareMathOperator{\coCechNv}{\mathsf{co\v{C}ech}}
\DeclareMathOperator{\cMap}{\mathcal{M}\mathsf{ap}}
\DeclareMathOperator*{\colim}{\mathsf{colim}}
\newcommand{\conn}{\mathsf{conn}}
\DeclareMathOperator{\Corr}{\mathsf{Corr}}

\newcommand{\defeq}{\overset{\mathsf{def}}{=}}
\newcommand{\dmod}{\mhyph\mathsf{mod}}
\DeclareMathOperator{\Dmod}{\mathsf{D}\dmod}
\DeclareMathOperator{\DGCat}{\mathsf{DGCat}}
\newcommand{\dR}{\mathsf{dR}}

\DeclareMathOperator{\Disk}{\mathsf{Disk}}

\DeclareMathOperator{\Eis}{\mathsf{Eis}}
\DeclareMathOperator{\Emb}{\mathsf{Emb}}
\newcommand{\En}{\mathsf{E}}
\newcommand{\EnAlg}[1]{\En_{#1}\!\mhyph\mathsf{alg}}
\newcommand{\enh}{\mathsf{enh}}
\DeclareMathOperator{\ev}{\mathsf{ev}}

\newcommand{\fin}{\mathsf{fin}}
\DeclareMathOperator{\Fun}{\mathsf{Fun}}

\DeclareMathOperator{\Gr}{\mathsf{Gr}}

\DeclareMathOperator{\Hecke}{\mathsf{H}}
\DeclareMathOperator{\HH}{\mathsfbf{H}}
\DeclareMathOperator{\Ho}{\mathsf{H}}

\DeclareMathOperator{\IndCoh}{\mathsf{IndCoh}}

\DeclareMathOperator{\LL}{\mathbb{L}}
\newcommand{\lfp}{\mathsf{lfp}}
\let\lim\relax
\DeclareMathOperator*{\lim}{\mathsf{lim}}
\DeclareMathOperator{\Loc}{\mathsf{Loc}}
\DeclareMathOperator{\LocSys}{\mathsf{LS}}

\newcommand{\mhyph}{\textsf{-}}

\DeclareMathOperator{\Mnfd}{\mathsf{Mnfd}}

\newcommand{\opp}{\mathsf{op}}

\newcommand{\perf}{\mathsf{perf}}
\DeclareMathOperator{\Perf}{\mathsf{Perf}}
\DeclareMathOperator{\PreStk}{\mathsf{PreStk}}
\newcommand{\pt}{\mathsf{pt}}

\DeclareMathOperator{\QCoh}{\mathsf{QCoh}}

\newcommand{\red}{\mathsf{red}}
\DeclareMathOperator{\Rep}{\mathsf{Rep}}

\DeclareMathOperator{\Shv}{\mathsf{Shv}}
\DeclareMathOperator{\Spc}{\mathsf{Spc}}
\DeclareMathOperator{\Sph}{\mathsf{Sph}}
\DeclareMathOperator{\Stk}{\mathsf{Stk}}
\newcommand{\sfD}{\mathsf{D}}
\newcommand{\spec}{\mathsf{spec}}
\DeclareMathOperator{\Spec}{\mathsf{Spec}}
\DeclareMathOperator{\Sym}{\mathsf{Sym}}

\DeclareMathOperator{\Tot}{\mathsf{Tot}}
\newcommand{\triv}{\mathsf{triv}}

\DeclareMathOperator{\Waki}{\mathsf{Waki}}

\mathchardef\mhyphensymb="2D

\newcommand{\arrdisp}{0.33ex}
\newcommand{\arrdisplacementsp}{0.72ex}

\newcommand{\ardis}{\ar@<\arrdisp>}
\newcommand{\ardissp}{\ar@<\arrdisplacementsp>}
\newenvironment{descriptionBox}{%
\smallskip%
\begin{adjustwidth}{50pt}{50pt}%
\justify%
\small}{%
\end{adjustwidth}%
\medskip%
}
\newcommand{\interior}[1]{\oset{\circ}{#1}}
\newcommand{\lbar}[1]{\overline{#1}}
\newcommand{\llrrb}[1]{\llbracket#1\rrbracket}
\newcommand{\llrrp}[1]{(\!(#1)\!)}

\newcommand{\oset}[3][0pt]{\mathrel{\ensurestackMath{\stackon[#1]{#3}{\scriptstyle #2}}}}

\newcommand{\wtilde}[1]{\widetilde{#1}}

\title{Eisenstein series via factorization homology of Hecke categories}
\author{Quoc P. Ho}
\address{Department of Mathematics, Hong Kong University of Science and Technology (HKUST), Clear Water Bay, Hong Kong}
\email{phuquocvn@gmail.com}
\author{Penghui Li}
\address{YMSC, Tsinghua University, Beijing, China}
\email{lipenghui@mail.tsinghua.edu.cn}
\date{\today}


\begin{document}
\begin{abstract}
Motivated by spectral gluing patterns in the Betti Langlands program, we show that for any reductive group $G$, a parabolic subgroup $P$, and a topological surface $M$, the (enhanced) spectral Eisenstein series category of $M$ is the factorization homology over $M$ of the $\En_2$-Hecke category $\Hecke_{G, P} = \IndCoh(\LocSys_{G, P}(D^2, S^1))$, where $\LocSys_{G, P}(D^2, S^1)$ denotes the moduli stack of $G$-local systems on a disk together with a $P$-reduction on the boundary circle.

More generally, for any pair of stacks $\mathcal{Y}\to \mathcal{Z}$ satisfying some mild conditions and any map between topological spaces $N\to M$, we define $(\mathcal{Y}, \mathcal{Z})^{N, M} = \mathcal{Y}^N \times_{\mathcal{Z}^N} \mathcal{Z}^M$ to be the space of maps from $M$ to $\mathcal{Z}$ along with a lift to $\mathcal{Y}$ of its restriction to $N$. Using the pair of pants construction, we define an $\En_n$-category $\Hecke_n(\mathcal{Y}, \mathcal{Z}) = \IndCoh_0\left(\left((\mathcal{Y}, \mathcal{Z})^{S^{n-1}, D^n}\right)^\wedge_{\mathcal{Y}}\right)$ and compute its factorization homology on any $d$-dimensional manifold $M$ with $d\leq n$,
\[
	\int_M \Hecke_n(\mathcal{Y}, \mathcal{Z}) \simeq \IndCoh_0\left(\left((\mathcal{Y}, \mathcal{Z})^{\partial (M\times D^{n-d}), M}\right)^\wedge_{\mathcal{Y}^M}\right),
\]
where $\IndCoh_0$ is the sheaf theory introduced by Arinkin--Gaitsgory and Beraldo. Our result naturally extends previous known computations of Ben-Zvi--Francis--Nadler and Beraldo.
\end{abstract}

\maketitle

\tableofcontents

\section{Introduction}
\subsection{Motivation}

This paper is motivated by the desire to produce the spectral category of the Betti Langlands program of Ben-Zvi and Nadler~\cite{ben-zvi_betti_2016} by gluing together categories over a disk. Roughly speaking, there are two types of gluing in the Betti Langlands program, which we call spectral Eisenstein gluing and spectral manifold gluing in this paper. The former concerns itself with building up the spectral category from parabolic inductions, a.k.a. Eisenstein series, along all standard parabolic subgroups, whereas the latter builds the spectral category by decomposing the underlying topological surface into simpler pieces.  

The main result of this paper provides the first step in this direction. More specifically, it says that Eisenstein series themselves admit manifold gluing in a strong sense: namely, they can be obtained by integrating (in the sense of factorization homology) the $\En_2$-Hecke categories $\Hecke_{G, P}$, new gadgets defined in this paper, over our topological surface $M$, i.e.
\[
	\int_M \Hecke_{G, P} \simeq \Eis_{G, P}(M),
\]
where $G$ is a reductive group over $\mathbb{C}$ and $P$ a parabolic subgroup.

The definitions of these objects and the precise statement will be given in \cref{subsec:main_results}. In \cref{subsubsec:geo_langlands_Eis_spectral} and \cref{subsubsec:Betti_langlands_spectral} below, we will briefly review the background and context regarding Eisenstein series and how they fit into the Langlands program. The reader who is already familiar with the subject can skip directly to \cref{subsec:main_results}.

\subsubsection{Geometric Langlands program} \label{subsubsec:geo_langlands_Eis_spectral}
Even though this paper is about the Betti Langlands program, for definiteness, we start with the geometric Langlands program since spectral Eisenstein gluing is better documented in the literature for the geometric Langlands program.\footnote{It is expected that the statements and proofs carry over to the Betti setting, at least for compact topological surfaces, i.e. closed surfaces without boundaries.} Let $C$ be any smooth proper algebraic curve and $G$ a reductive group with the Langlands dual $G^\vee$. The latest incarnation of the geometric Langlands program, as pioneered by Arinkin and Gaitsgory~\cite{arinkin_singular_2015} building on previous work of Beilinson and Drinfel'd~\cite{beilinson_quantization_1991}, asserts that we have an equivalence of (DG-)categories
\[
    \IndCoh_\mathcal{N}(\LocSys_G) \simeq \Dmod(\Bun_{G^\vee}). \teq \label{eq:geometric_Langlands}
\]
Here, $\IndCoh_\mathcal{N}(\LocSys_G)$ denotes the category of ind-coherent sheaves with nilpotent singular support on the moduli stack $\LocSys_G$ of \emph{de Rham} $G$-local systems on $C$ and $\Dmod(\Bun_{G^\vee})$ denotes the category of $\mathsf{D}$-modules on the moduli stack $\Bun_{G^\vee}$ of $G$-principal bundles on $C$. The two sides of~\eqref{eq:geometric_Langlands} are usually called the spectral and automorphic sides, respectively.

The equivalence~\eqref{eq:geometric_Langlands} is not merely an abstract equivalence between DG categories. Rather, in a precise sense, it is supposed to be compatible with Hecke operators and Eisenstein series. In fact, it was the compatibility with Eisenstein series which led Arinkin and Gaitsgory to consider $\IndCoh_{\mathcal{N}}(\LocSys_G)$ rather than $\QCoh(\LocSys_G)$ for the spectral side. The latter, a.k.a. the tempered part of the spectral category, is too small to match with the automorphic side.

Let us now comment on one important aspect of $\IndCoh_{\mathcal{N}}(\LocSys_G)$ regarding Eisenstein series. For any parabolic subgroup $P$ of $G$ with Levi subgroup $L$, Eisenstein series on the spectral side is given by pulling and pushing along the following correspondence
\[
    \LocSys_L \leftarrow \LocSys_P \rightarrow \LocSys_G.
\]
The virtue of the full subcategory $\IndCoh_\mathcal{N}(\LocSys_G) \subseteq \IndCoh(\LocSys_G)$ is that it is spanned precisely by the images of all $\QCoh(\LocSys_L)$ under Eisenstein series, where $L$ runs over all standard Levi subgroups, including $L = G$. In fact, since the image of $\QCoh(\LocSys_L)$ in $\QCoh(\LocSys_P)$ under pullback generates the target, $\IndCoh_\mathcal{N}(\LocSys_G) \subseteq \IndCoh(\LocSys_G)$ is, equivalently, generated by the images of $\QCoh(\LocSys_P)$ under pushforward functors.

A stronger statement is true. The spectral Eisenstein gluing conjecture of~\cite{gaitsgory_outline_2015}, now a theorem by~\cites{arinkin_category_2017-1,beraldo_spectral_2020} states that, very roughly, the category $\IndCoh_\mathcal{N}(\LocSys_G)$ could be obtained by gluing together $\QCoh(\LocSys_P)$, i.e. it can be realized as a limit of categories whose terms roughly look like $\QCoh(\LocSys_P)$.\footnote{For this to actually work, one needs to replace $\QCoh(\LocSys_P)$ by the so-called enhanced Eisenstein categories, which we will turn to shortly. See also \cref{rmk:on_the_name_Eis}.} Using this observation,~\cite{gaitsgory_outline_2015} suggests that one might try to prove the Geometric Langlands conjecture (or at least, to produce a functor from the spectral side to the automorphic side) by gluing together functors of the form
\[
    \QCoh(\LocSys_L) \to \Dmod(\Bun_{L^\vee}).
\]
Such a functor has indeed been constructed using, for example, Beilinson's spectral projector~\cite{gaitsgory_outline_2015}*{Theorem 4.5.2}. In the Betti setting, this is also done in~\cite{nadler_spectral_2019}.

\subsubsection{Betti Langlands program} \label{subsubsec:Betti_langlands_spectral}
In what follows, since we work in the topological context, we use $M$ to denote a Riemann surface. The Betti Langlands program of Ben-Zvi and Nadler~\cite{ben-zvi_betti_2016} is a topological analog of the above.  Namely, it asserts that we have an equivalence of categories
\[
    \IndCoh_\mathcal{N}(\LocSys_G) \simeq \Shv_{\mathcal{N}^\vee}(\Bun_{G^\vee}), \teq\label{eq:Betti_Langlands}
\]
which is also compatible with Hecke operators and Eisenstein series. Here (and in the remainder of the paper), $\LocSys_G$ is to be understood as the moduli stack of Betti $G$-local systems on the \emph{topological space} underlying $M$. In addition to spectral Eisenstein gluing, however, the spectral side of the Betti Langlands program affords spectral manifold gluing induced by building up $M$ from more elementary pieces such as cylinders and pairs of pants etc., see~\cite{ben-zvi_betti_2021}. Note, however, that it is not yet known whether it is possible to build the whole spectral category from just categories over a disk.

\begin{rmk} \label{rmk:two_ways_to_glue}
Unfortunately, these two kinds of gluing are both called spectral gluing in the literature. In other words, the papers~\cite{ben-zvi_betti_2021} and~\cite{gaitsgory_outline_2015} use the term spectral gluing to refer to completely different phenomena. To avoid confusion, we will refer to the gluing done by~\cite{ben-zvi_betti_2021} spectral manifold gluing and the one done in~\cites{gaitsgory_outline_2015,arinkin_category_2017-1,beraldo_spectral_2020} spectral Eisenstein gluing.
\end{rmk}

\subsection{The goal of this paper}
With these two ways of gluing available, it is tempting to try to prove the Betti Langlands conjecture by building up from Eisenstein series over a disk. As the first step, it is natural to ask how spectral Eisenstein gluing interacts with spectral manifold gluing.

The goal of this paper is to show that Eisenstein series themselves also admit manifold gluing. More precisely, we show that Eisenstein series for a Riemann surface $M$ can be glued together using topological factorization homology of $M$ with coefficients in the so-called $\En_2$-Hecke categories. This shows, in particular, that unlike the whole spectral category, Eisenstein series themselves can be built up from just a disk.

\subsection{The main results} \label{subsec:main_results}
We will now describe our results more precisely. The main technical tools we use to formulate gluing are the theory of topological factorization homology, as developed by Lurie and Ayala--Francis~\cites{lurie_higher_2017,ayala_factorization_2015} and the hybrid sheaf theory $\IndCoh_0$ (a mixture between $\IndCoh$ and $\QCoh$) appearing in the work of Arinkin--Gaitsgory and Beraldo~\cites{arinkin_category_2017-1,beraldo_spectral_2020}. The reader can find a summary of these theories in \cref{subsec:prelim_IndCoh_0} and \cref{subsec:prelim_fact_hom}.

\subsubsection{Hecke categories and (enhanced) Eisenstein series} \label{subsubsec:intro:Hecke_cat}
For any standard parabolic subgroup $P \subseteq G$ and a map between topological spaces $N \to M$, we will use $\LocSys_{G, P}(M, N)$ to denote the moduli stack of $G$-local systems on $M$ along with a $P$-structure on its pullback to $N$. Similarly, $\LocSys_G(M)$ is the moduli stack of $G$-local systems on $M$.

We define an $\En_2$-category $\Hecke_{G, P}$, i.e. a braided monoidal category, whose underlying category is given by
\[
	\Hecke_{G, P} = \IndCoh(\LocSys_{G, P}(D^2, S^1)),
\]
where $\LocSys_{G, P}(D^2, S^1)$ is the moduli space of of $G$-local systems on a two-dimensional disk along with a $P$-reduction along the boundary circle. The $\En_2$-monoidal structure comes from the pair of pants construction.

\begin{rmk}
The name Hecke is motivated by the same construction but one dimensional lower, i.e. by replacing the pair $S^1 \hookrightarrow D^2$ with $S^0 = * \sqcup * \hookrightarrow D^1$. Indeed, such a construction yields the quasi-coherent Hecke category
\[
	\IndCoh(B\backslash G/B) \simeq \QCoh(B\backslash G/B),
\]
where the tensor product is given by the usual convolution diagram. Replacing quasi-coherent sheaves by $\sfD$-modules or $\ell$-adic sheaves, we obtain the finite Hecke category which plays an important role in the theory of character sheaves and HOMFLY-PT knot homology~\cites{bezrukavnikov_character_2012,ben-zvi_character_2009,webster_geometric_2017,shende_legendrian_2017}.
\end{rmk}

More generally, for any topological surface $M$ with (possibly non-empty) boundary $\partial M$, we define the (enhanced spectral) Eisenstein category
\[
	\Eis_{G, P}(M) \defeq \IndCoh_0(\LocSys_{G, P}(M, \partial M)^\wedge_{\LocSys_P(M)}).
\]
In particular, when $\partial M = \emptyset$,
\[
	\Eis_{G, P}(M) \defeq \IndCoh_0(\LocSys_G(M)^\wedge_{\LocSys_P(M)}).
\]
Note that when $M$ is a two-dimensional disk, we recover $\Hecke_{G, P}$, i.e. $\Eis_{G, P}(D^2) \simeq \Hecke_{G, P}$. See \cref{subsec:prelim_IndCoh_0} for a quick review of the theory $\IndCoh_0$.

\begin{rmk} \label{rmk:boundary}
By definition, topological manifolds appearing in this paper are without boundary in the usual sense. However, when $M$ is non-compact, one can make sense of what it means to take its boundary $\partial M$ by using a compactification $\lbar{M}$ of $M$ and set $\partial M \defeq \partial \lbar{M}$. Of course, when $M$ is already compact, $\partial M = \emptyset$. See \cref{subsubsec:boundary_of_manifolds} for a more detailed discussion. In this paper, the term boundary is strictly used in the sense above, i.e., not the usual sense. In particular, non-compact manifolds have non-empty boundary in this sense.

For example, the two dimensional disk $\mathbb{R}^2$ has no boundary in the usual sense but in our convention, its boundary $\partial \mathbb{R}^2 \simeq S^1$ is a circle. Similarly, $\partial (S^1 \times \mathbb{R}) \simeq S^1 \sqcup S^1$.
\end{rmk}

The following statement is a special, but most interesting, case of our main result.

\begin{thm}[\cref{cor:main_Eis}] \label{thm:intro:main_Eis}
For any topological surface $M$ (with possibly non-empty boundary), we have
\[
	\int_M \Hecke_{G, P} \simeq \Eis_{G, P}(M).
\]
\end{thm}

\begin{rmk} \label{rmk:on_the_name_Eis}
When $\partial M = \emptyset$, the category $\Eis_{G, P}(M)$ is precisely the topological analog of the categories appearing in the spectral Eisenstein gluing conjecture/theorem which also go under the name \emph{parabolic categories} and under various notations $F_P\dmod(\QCoh(\LocSys_P(M)))$ and $\QCoh(\LocSys_P(M))_{\conn / \LocSys_G(M)}$ in~\cites{gaitsgory_outline_2015,arinkin_singular_2015,beraldo_spectral_2020}. The category
\[
	\Eis_{G, P}(M) = \IndCoh(\LocSys_G(M)^\wedge_{\LocSys_P(M)}) \times_{\IndCoh(\LocSys_P(M))} \QCoh(\LocSys_P(M)) 
\]
is a full subcategory of $\IndCoh_{\mathcal{N}_P}(\LocSys_G(M)^\wedge_{\LocSys_P(M)})$, the source of the enhanced spectral Eisenstein series functor $\Eis_{P,\spec}^{\enh}$~\cite{gaitsgory_outline_2015}*{\S6.5.8}. Thus, the category $\Eis_{G, P}(M)$ can be thought of as the spectral category of enhanced \emph{tempered} Eisenstein series.

We note that the categories $\Eis_{G, P}(M)$ naturally map to the spectral category $\IndCoh_\mathcal{N}(\LocSys_G)$. Moreover, for each $P$, the full subcategory generated by the image of $\Eis_{G, P}(M)$ in the spectral category coincide with the full subcategory generated by the image of the usual Eisenstein functor $\QCoh(\LocSys_L) \to \IndCoh_\mathcal{N}(\LocSys_G)$.
\end{rmk}

\subsubsection{The case of non-compact surfaces}
We saw above that even though $\Hecke_{G, P}$ is defined via $\IndCoh$, $\IndCoh_0$ naturally shows up when we integrate $\Hecke_{G, P}$ over a surface $M$ to obtain $\Eis_{G, P}(M)$. When $M$ has non-empty boundary (see \cref{rmk:boundary}), however, the situation simplifies and we have the following result.

\begin{thm}[\cref{prop:closed_embedding_non-compact_M,thm:Eis_non-compact}]
\label{thm:intro_Eis_non-compact}
Let $M$ be a non-compact manifold. $\LocSys_P(M) \to \LocSys_{G, P}(M, \partial M)$ is a closed embedding. Thus,
\[
	\Eis_{G, P}(M) \simeq \IndCoh(\LocSys_{G, P}(M, \partial M)^\wedge_{\LocSys_P(M)}) \xhookrightarrow{\text{f.f.}} \IndCoh(\LocSys_{G, P}(M, \partial M)),
\]
where f.f. stands for fully faithful. In other words, $\Eis_{G, P}(M)$ is the full subcategory
\[
	\IndCoh_{\LocSys_P(M)}(\LocSys_{G, P}(M, \partial M)) \subseteq \IndCoh(\LocSys_{G, P}(M, \partial M))	
\]
consisting of ind-coherent sheaves set-theoretically supported on $\LocSys_P(M)$.
\end{thm}

\begin{expl}
\label{expl:intro_cylinder}
Plugging $M=S^1 \times \mathbb{R}$ to the theorem above, we obtain the trace of our $\En_2$-Hecke category $\Hecke_{G, P}$:
\[
	\int_{S^1} \Hecke_{G,P} \simeq \Eis_{G,P}(S^1 \times \mathbb{R}) \simeq \IndCoh_{P/P}(P/P \times_{G/G} P/P),
\]
where $\IndCoh_{P/P}(P/P \times_{G/G} P/P)$ denotes the full subcategory of $\IndCoh(P/P \times_{G/G} P/P)$ consisting of ind-coherent sheaves with set-theoretic support on $P/P$. Note that all the quotients that appear here are with respect to the conjugation actions.

Being the trace of an $\En_2$-category, the left hand side has an natural induced $\En_1$-structure, which is identified with the convolution $\En_1$-structure on the right hand side. In particular, we have an $\En_1$-monoidal functor from the $\En_2$-Hecke category $\Hecke_{G,B}$ to the affine Hecke category $\Hecke_{\aff}$
\[
	\Hecke_{G,B} \to \int_{S^1} \Hecke_{G,B} \simeq  \IndCoh_{B/B}(B/B \times_{G/G} B/B) \hookrightarrow \IndCoh(B/B \times_{G/G} B/B) \defeq \Hecke_{\aff}.
\]
\end{expl}

\subsubsection{A generalization} \label{subsubsec:intro:a_generalization}
The pair $BP \to BG$ used in \cref{thm:intro:main_Eis} can be replaced by any pair of stacks $\mathcal{Y} \to \mathcal{Z}$ such that both $\mathcal{Y}$ and $\mathcal{Z}$ are perfect and locally of finite presentation. Before stating the result, we will need to introduce some notation.

For any stack $\mathcal{Y}$ and any topological space $M$, we use $\mathcal{Y}^M \defeq \cMap(M, \mathcal{Y})$ to denote the associated (derived) mapping stack, i.e. the stack of maps from $M$ to $\mathcal{Y}$, where we view $M$ as a constant stack. For example, when $\mathcal{Y} = BG$, $\mathcal{Y}^M = BG^M = \LocSys_G(M)$ seen above. In general, when $M$ is a finite CW complex, we can build $\mathcal{Y}^M$ iteratively using a cell attachment presentation of $M$.

\begin{expl}
For any stack $\mathcal{Y}$, $\mathcal{Y}^{S^1} \simeq \mathcal{Y} \times_{\mathcal{Y} \times \mathcal{Y}} \mathcal{Y}$ and $\mathcal{Y}^{S^2} \simeq \mathcal{Y} \times_{\mathcal{Y}^{S^1}} \mathcal{Y}$. These come from the following presentations $S^1 \simeq \pt \sqcup_{\pt \sqcup \pt} \pt$ and $S^2 \simeq \pt \sqcup_{S^1} \pt$, respectively.
\end{expl}

Now, for any pair of stacks $\mathcal{Y} \to \mathcal{Z}$ and any pair of topological spaces $N \to M$, we use $(\mathcal{Y}, \mathcal{Z})^{N, M} = \mathcal{Y}^N \times_{\mathcal{Z}^N} \mathcal{Z}^M$ to denote the stack of commutative squares
\[
\begin{tikzcd}
	N \ar{r} \ar{d} & \mathcal{Y} \ar{d} \\
	M \ar{r} & \mathcal{Z}
\end{tikzcd}
\]
When $\mathcal{Y} \to \mathcal{Z}$ is chosen to be $BP \to BG$, we recover $(BP, BG)^{N, M} = \LocSys_{G, P}(M, N)$ mentioned above.

Given such a pair $\mathcal{Y} \to \mathcal{Z}$, we can define the $\En_n$-Hecke category 
\[
	\Hecke_n(\mathcal{Y}, \mathcal{Z}) \defeq \IndCoh_0\left(\left((\mathcal{Y}, \mathcal{Z})^{S^{n-1}, D^n}\right)^\wedge_{\mathcal{Y}}\right),
\]
whose $\En_n$-monoidal structure is given by a higher dimensional analog of the pair of pants construction. We obtain the following generalization of \cref{thm:intro:main_Eis} above.

\begin{thm}[\cref{thm:main}] \label{thm:intro:general_facthom_n}
Let $\mathcal{Y} \to \mathcal{Z}$ be a morphism of stacks such that $\mathcal{Y}$ and $\mathcal{Z}$ are perfect and locally of finite presentation. Then, for any $n$-dimensional manifold $M$, we have
\[
	\int_M \Hecke_n(\mathcal{Y}, \mathcal{Z}) \simeq \IndCoh_0\left(\left((\mathcal{Y}, \mathcal{Z})^{\partial M, M}\right)^\wedge_{\mathcal{Y}^M}\right).
\]
\end{thm}

Given an $\En_n$-algebra, one can take its factorization homology over any $d$-dimensional manifold where $d\leq n$. The following is an immediate consequence of the theorem above.

\begin{cor} \label{cor:intro:general_facthom_d}
Let $\mathcal{Y} \to \mathcal{Z}$ be as in the previous theorem. Then, for any $d$-dimensional manifold $M$, with $d\leq n$, we have
\begin{align*}
	\int_{M} \Hecke_n(\mathcal{Y}, \mathcal{Z}) 
	\simeq \int_{M \times D^{n-d}} \Hecke_n(\mathcal{Y}, \mathcal{Z}) 
	\simeq \IndCoh_0\left(\left((\mathcal{Y}, \mathcal{Z})^{\partial (M\times D^{n-d}), M}\right)^\wedge_{\mathcal{Y}^M}\right).
\end{align*}
\end{cor}

\subsubsection{The Hecke pair condition} \label{subsubsec:Hecke_pair_condition}
We note that the sheaf theory $\IndCoh_0$ for unbounded stacks is much more complicated than the theory for bounded stacks, in terms of definition, computability, and formalism. This is unavoidable if we work with high dimensional manifolds, even if we start with a smooth stack. For example, the derived mapping scheme $(\mathbb{A}^1)^{S^2} \simeq \mathbb{A}^1 \times \Spec \Sym \mathbb{C}[2] \simeq \mathbb{A}^1 \times \mathbb{A}^1[-2]$ is unbounded. 

While this does not affect the proof of \cref{thm:intro:general_facthom_n} (since we do not need any special property of $\IndCoh_0$ for bounded stacks), in practice, it is in general much easier to stay in the world of bounded stacks as far as $\IndCoh_0$ is concerned. Fortunately, our main example $BP \to BG$ used in Eisenstein series satisfies a certain finiteness condition called the \emph{Hecke pair} condition (see \cref{defn:Hecke_pair}). For any Hecke pair $\mathcal{Y} \to \mathcal{Z}$, the proof of \cref{thm:intro:general_facthom_n} stays within the world of perfect, locally of finite type, and bounded stacks. The reader whose main interest is Eisenstein series can restrict themselves to this case without losing the main point of the paper.

\subsection{Relation to other work}

\subsubsection{Betti Langlands}
It is proved in \cite{ben-zvi_integral_2010} that for any topological surface $M$,
\[
	\int_M \Rep(G) \simeq \QCoh(\LocSys_G(M)). \teq\label{eq:Betti_Langlands_LS_FactHom}
\]
Here, $\Rep(G) = \QCoh(BG)$ is a symmetric monoidal category (i.e. an $\En_\infty$-category), viewed as an $\En_2$-category when taking factorization homology. Our \cref{thm:intro:main_Eis} recovers this statement when $P=G$. Indeed, in this case, we have
\[
	\Hecke_{G, G} = \IndCoh(\LocSys_{G, G}(D^2, S^1)) \simeq \IndCoh(\LocSys_G(D^2)) \simeq \QCoh(BG) \simeq \Rep(G)
\]
and
\[
	\Eis_{G, G}(M) = \IndCoh_0(\LocSys_{G, G}(M, \partial M)^\wedge_{\LocSys_G(M)}) \simeq \IndCoh_0(\LocSys_G(M)^\wedge_{\LocSys_G(M)}) \simeq \QCoh(\LocSys_G(M)).
\]

\subsubsection{}
This result is refined and extended in~\cite{beraldo_topological_2019} where for any lfp stack $\mathcal{Y}$ and a fixed positive integer $n$, the $\En_n$-spherical category $\Sph(\mathcal{Y}, n-1) = \IndCoh_0((\mathcal{Y}^{S^{n-1}})^\wedge_\mathcal{Y})$ is defined and its factorization homology on any $d$-manifold $M$ is computed
\[
	\int_M \Sph(\mathcal{Y}, n-1) \simeq \IndCoh_0((\mathcal{Y}^{\partial(M\times D^{n-d})})^\wedge_{\mathcal{Y}^{M}}).
\]
This can be recovered by setting $\mathcal{Z} = \pt$ in \cref{cor:intro:general_facthom_d}.

Our result is thus a common generalization of both of these results in \cites{ben-zvi_integral_2010,beraldo_topological_2019}.

\subsubsection{Geometric Langlands}
In the Geometric Langlands program, we have an analog of~\cref{eq:Betti_Langlands_LS_FactHom}. However, instead of an equivalence, the (de Rham version of the) RHS of~\cref{eq:Betti_Langlands_LS_FactHom} only embeds fully faithfully into a factorization category, an analog of the LHS of~\cref{eq:Betti_Langlands_LS_FactHom}.

We learned from D. Beraldo that a factorization category analogous to the $\En_2$-category $\Hecke_{G, \pt}$ appeared in Rozenblyum's thesis~\cite{rozenblyum_connections_2011}. The relation between this and the de Rham analog of our main result is also mentioned in the introduction of~\cite{rozenblyum_connections_2021}.

\subsection{Questions and future work}
Our result opens up several questions that we hope to address in future publications. 

\subsubsection{Interaction with Ben-Zvi--Nadler's spectral manifold gluing}
In~\cite{ben-zvi_betti_2021}, a spectral manifold gluing formula is formulated and proved for the \emph{entire spectral category} for any $2$-dimensional (possibly open) manifold. One important difference between our gluing and theirs is that while Eisenstein series categories could be built up from just their values on a $2$-dimensional disk using collar gluing, the entire spectral category has to be built up from more complicated pieces and one can only glue along cylinders. 

From the perspective of spectral Eisenstein gluing,\footnote{i.e. gluing Eisenstein series categories together to obtain the whole spectral category.} it is natural to investigate the relation between our gluing and theirs. In particular, we would like to
\begin{myenum}{(\roman*)}
	\item extend spectral Eisenstein series gluing to non-compact manifolds; and
	\item understand local-to-global properties of this gluing. We expect that the gluing seen in~\cite{ben-zvi_betti_2021} is a manifestation of this.
\end{myenum}
We expect that the solution to these problems will play an important role in the construction/study of the conjectural $(3+1)$-dimensional TFT suggested by Ben-Zvi--Nadler~\cite{ben-zvi_betti_2021}*{\S1.1.1} which assigns the spectral category $\IndCoh_\mathcal{N}(\LocSys_G(M))$ to a topological surface $M$. 

\subsubsection{Automorphic expectations}
We describe here some expectations on the automorphic side of the Betti Langlands program. For simplicity, we shall focus on the case where $P=B$. 

Let $G^\vee$ be the Langlands dual group of $G$, $G^{\vee}\llrrp{z}$, $G^\vee\llrrb{z}$ the loop and arc groups of $G^\vee$ respectively. Let $B^\vee$ be the dual Borel, $I^\vee \subset G^\vee\llrrb{z}$ the Iwahori subgroup associated to $B^\vee$, and $\Gr_{G^\vee} = G^\vee\llrrp{z}/G^\vee\llrrb{z}$ the affine Grassmannian attached to the group $G^\vee$. Denote by $\Shv_{I^\vee}(\Gr_{G^\vee})$ and $\Shv_{N^\vee\llrrp{z}}(\Gr_{G^\vee})$ the categories of $I$- and $N^\vee\llrrp{z}$-constructible sheaves on $\Gr_{G^\vee}$ respectively. Combining~\cites{arkhipov_quantum_2004, raskin_chiral_2014, gaitsgory_semi-infinite_2018}, we have the following statement.

\begin{thm} \label{thm:quantum_connection}
There are equivalences of DG categories\footnote{The first equivalence is due to~\cites{raskin_chiral_2014, gaitsgory_semi-infinite_2018} whereas the second is due to \cite{arkhipov_quantum_2004}. The last equivalence is a simple computation, see~\cref{eq:computation_of_H_GP}.}
\[
	\Shv_{N^\vee\llrrp{z}}(\Gr_{G^\vee}) \simeq \Shv_{I^\vee}(\Gr_{G^\vee}) \simeq \IndCoh(\mathfrak{n}^*[-1]/B) \simeq \Hecke_{G, B}.
\]
\end{thm}

$\Shv_{N^\vee\llrrp{z}}(\Gr_{G^\vee})$ is naturally a factorizable category and hence, is an $\En_2$-category. It is thus natural to expect the following statement.

\begin{expt}
The equivalence of \cref{thm:quantum_connection} is compatible with the $\En_2$-structure on both sides; namely, we have an equivalence of $\En_2$-categories $\Shv_{N^\vee\llrrp{z}}(\Gr_{G^\vee}) \simeq \Hecke_{G, B}$.
\end{expt}

For this reason, we use $\Hecke^{\auto}_{G^\vee,B^\vee}$ to denote the $\En_2$-category $\Shv_{N^\vee\llrrp{z}}(\Gr_{G^\vee})$. 

\subsubsection{}
Let us now consider the automorphic side. Let $\wtilde{M}$ be a compact Riemann surface and 
 $S \subset \wtilde{M}$ be a finite set. Put $M=\wtilde{M}\setminus S$ and let $\Bun_{G^\vee,N^\vee}=\Bun_{G^\vee,N^\vee}(\wtilde{M},S)$ be the moduli stack of $G^\vee$-bundles on $\wtilde{M}$  with $N^\vee$-reduction along $S$. Denote by $\mathcal{N}^\vee \subset T^*\Bun_{B^\vee,N^\vee}$ the global nilpotent cone. Recall that the Betti Langlands conjecture \cite{ben-zvi_betti_2016} asserts an equivalence
 \[
	 \Shv_{\mathcal{N}^\vee}(\Bun_{G^\vee,N^\vee}(\wtilde{M},S))
 	\simeq \IndCoh_{\mathcal{N}}(\LocSys_{G,B}(M, \partial M)).
 \]
 
 \subsubsection{}
 We will now consider the automorphic Eisenstein series for $B$. Denote by $\Eis_{G^\vee,B^\vee}^{\auto}(\wtilde{M},S)$ the full subcategory of $\Shv(\Bun_{G^\vee,N^\vee})$ generated by the image of $\Loc(\Bun_{T^\vee,1})$ under the functor $p_!q^*$ (which is expected to lie in $\Shv_{\mathcal{N}^\vee}(\Bun_{G^\vee,N^\vee})$)\footnote{It is a subcategory rather than something more sophisticated in view of \cref{thm:intro_Eis_non-compact}.}
\[
\begin{tikzcd}
 \Bun_{T^\vee,1}  & \Bun_{B^\vee,N^\vee} \ar{l}[swap]{q} \ar{r}{p}  &  \Bun_{G^\vee,N^\vee}
\end{tikzcd}
\]
where $\Loc(\Bun_{T^\vee,1}) \subset \Shv(\Bun_{T^\vee,1})$ denotes the full subcategory spanned by local systems. The compatibility between Betti Langlands conjecture and Eisenstein series asserts that $\Eis^{\auto}_{G^\vee,B^\vee}(\wtilde{M},S) \simeq \Eis_{G,B}(M)$. We therefore expect the following automorphic version of \cref{thm:intro:main_Eis} for non-compact $M$.

\begin{expt}
	\label{expt:intro_auto_eis}
	Assume $S$ is non-empty, then 
		 	\[
		\int_{M} \Hecke^{\auto}_{G^\vee, B^\vee} 
		\simeq \Eis^{\auto}_{G^\vee,B^\vee}(\wtilde{M},S)
		\]
\end{expt}

\begin{rmk}
For $S=\emptyset$, one can formulate the above expectation by using Betti analogue of the enhanced Eisenstein series defined in \cite{gaitsgory_outline_2015}. 
\end{rmk}

\begin{expl}
For $\wtilde{M}=\mathbb{P}^1$ and $S=\{0\}$, \cref{expt:intro_auto_eis} is given by the composition of the first equivalence in \cref{thm:quantum_connection} and the Radon transform $\Shv_{I^\vee}(\Gr_{G^\vee}) \simeq \Shv_{\mathcal{N}^\vee}(\Bun_{B^\vee,N^\vee}(\mathbb{P}^1,0))$.
\end{expl}

\begin{expl}
The case for $\wtilde{M}=\mathbb{P}^1$ and $S=\{0,\infty\}$ can be related to Langlands duality for affine Hecke categories. Denote by $ \wtilde{\mathsf{Fl}}_{G^\vee}= G^\vee\llrrp{z}/I^\vee_0,$ for $I^\vee_0 \subset I^\vee$ the pro-unipotent radical. Denote by  $\Hecke^{\auto}_{\aff}$ the category of $I^\vee$-constructible sheaves on $\wtilde{\mathsf{Fl}}_{G^\vee}$; it is naturally an $\En_1$-category by convolution. The (twisted) Radon transform \cite{nadler_geometric_2019-1}*{Lemma 2.6.1} yields an equivalence
\[
	\Shv_{\mathcal{N}^\vee}(\Bun_{G^\vee,N^\vee}(\mathbb{P}^1,\{0,\infty\})) \simeq \Hecke^{\auto}_{\aff}.
\]

Under this equivalence, the full subcategory $\Eis_{G^\vee,B^\vee}(\mathbb{P}^1,\{0,\infty\})$ is identified with the full $\En_1$-subcategory $\Waki \subset \Hecke^{\auto}_{\aff}$ generated by (universal-monodromic) Wakimoto sheaves. Therefore, \cref{expt:intro_auto_eis} implies the following equivalence
\[
	\int_{S^1} \Hecke^{\auto}_{G,B} \simeq \Waki.
\]

\subsubsection{}
The equivalence above should be compatible with the natural $\En_1$-structure on both sides. In view of \cref{expl:intro_cylinder}, the following diagram of $\En_1$-categories is expected to commute
 \[
 \begin{tikzcd}
 \Hecke_{G,B} \ar{r} \ar{d}{\simeq} &	\int_{S^1} \Hecke_{G,B} \ar{r}{\simeq} \ar{d}{\simeq} & \IndCoh_{B/B}(B/B \times_{G/G} B/B) \ar[hookrightarrow]{r}\ar{d}{\simeq} & \Hecke_{\aff} \ar{d}{\simeq} \\
  \Hecke^{\auto}_{G,B} \ar{r} &	\int_{S^1} \Hecke^{\auto}_{G,B} \ar{r}{\simeq} & \Waki  \ar[hookrightarrow]{r}	&  \Hecke^{\auto}_{\aff}
 \end{tikzcd}
 \]
 where the first two vertical arrows are induced by \cref{expt:intro_auto_eis}, and last two vertical arrows are universal-monodromic version of Bezrukavnikov's Langlands duality for affine Hecke categories~\cite{bezrukavnikov_two_2016}.
\end{expl}

\section{Preliminaries}

We will set up the necessary notation and review results used throughout the paper. We will mainly follow the notation and conventions of~\cite{gaitsgory_study_2017}; most results about category theory and algebraic geometry that we use in the paper can be found there.

\subsection{Category theory} \label{subsec:prelim_cat_theory}
Throughout the paper, the term \emph{DG category} means \emph{stable presentable $k$-linear $\infty$-category} in the sense of~\cite{lurie_higher_2017}, where $k$ is a fixed algebraically closed field of characteristic $0$. We will use $\DGCat$ to denote the category of DG categories with morphisms given by continuous functors. $\DGCat$ is equipped with the Lurie symmetric monoidal structure.

We use $\Spc$ to denote the $\infty$-category of spaces, or equivalently, $\infty$-groupoids. Moreover, $\Spc_{\fin}$ is the full subcategory of $\Spc$ spanned by finite CW complexes. Both of these categories are symmetric monoidal under the usual Cartesian products of spaces.

\subsection{Derived algebraic geometry} \label{subsec:prelim_DAG}

We will now review some notions from derived algebraic geometry. Throughout the paper, we work over a fixed algebraically closed field $k$ of characteristic $0$. All of our schemes/stacks are, by default, derived. We thus drop the adjective derived from the terminology.

The various technical properties of schemes/prestacks/stacks recalled here are only necessary because they are required by the theories of $\IndCoh$ and $\IndCoh_0$ used in the paper. As such, the reader who is unfamiliar with the theory may simply skim this section to get the general idea (and return to it when necessary) without losing the gist of the paper.

\subsubsection{Affine schemes}
Let $\Aff$ denote the $\infty$-category of affine schemes over $k$. It is the opposite of the category of (DG) commutative rings over $k$ cohomologically supported in degrees $\leq 0$. An affine scheme $\Spec A$ is \emph{bounded} or \emph{eventually co-connective} if $A$ is supported in finitely many cohomological degrees. We use $\Aff^{< \infty}$ to denote the full subcategory of $\Aff$ consisting of bounded affine schemes.

Let $\Aff_{\aft} \subseteq \Aff$ denote the full subcategory of $\Aff$ consisting of affine schemes $\Spec A$ \emph{almost of finite type}, which means that $\Ho^0(A)$ is of finite type over $k$ and for any $i$, $\Ho^i(A)$ is a finitely generated $A$-module.

\subsubsection{Prestacks}
The category of prestacks is defined to be the category of functors from $\Aff^{\opp}$ to $\Spc$. Namely,
\[
	\PreStk \defeq \Fun(\Aff^{\opp}, \Spc).
\]

\subsubsection{Stacks} \label{subsubsec:prelim_stacks}
Let $\Stk \subseteq \PreStk$ be the full subcategory consisting of quasi-compact algebraic stacks with affine diagonal and with an atlas in $\Aff_{\aft}$. We simply call them stacks.

A stack $\mathcal{Y} \in \Stk$ is \emph{bounded} if for some (equivalently, any) atlas $Y \to \mathcal{Y}$ where $Y\in \Aff_{\aft}$, $Y$ is in fact in $\Aff_{\aft}^{<\infty}$, i.e. it is bounded. We let $\Stk^{<\infty}$ denote the full subcategory of $\Stk$ consisting of bounded stacks.

Note that boundedness is not generally preserved under fiber products. A morphism $\mathcal{Y} \to \mathcal{Z}$ in $\Stk$ is bounded if its base change to any $S \to \mathcal{Z}$ where $S \in \Aff_{\aft}^{< \infty}$ is bounded.

A stack $\mathcal{Y} \in \Stk$ is \emph{perfect} if $\QCoh(\mathcal{Y})$ is generated by its subcategory of perfect complexes $\Perf(\mathcal{Y})$.\footnote{Note that our stacks already have affine diagonals by convention.} This notion was introduced by Ben-Zvi--Francis--Nadler in~\cite{ben-zvi_integral_2010}. 

A stack $\mathcal{Y} \in \Stk$ is \emph{locally finitely presented} (lfp) if its cotangent complex $\LL_\mathcal{Y} \in \QCoh(\mathcal{Y})$ is perfect.

We will generally use these properties to decorate $\Stk$ to denote the full subcategory consisting of stacks satisfying all of these properties. In particular, we use $\Stk^{<\infty}_{\perf, \lfp} \subseteq \Stk$ to denote the full category consisting of bounded, perfect, and locally finitely presented stacks. 

\subsection{The theory of $\IndCoh_0$} \label{subsec:prelim_IndCoh_0}

The sheaf theory $\IndCoh_0$ is developed in~\cite{arinkin_category_2017-1} in the bounded case and in~\cite{beraldo_center_2020} more generally. It plays an important role in the formulation and proof of the spectral Eisenstein gluing theorem. We will now briefly recall what we need about the theory and refer the reader to~\cite{beraldo_center_2020} and references therein for proofs.

We note that the theory of $\IndCoh_0$ for bounded stacks is much simpler than the general case. Fortunately, this is all that we need for Eisenstein series. In what follows, with a view toward generalizations beyond Eisenstein series, we will, however, try to include the more general case while at the same time include remarks about the simplifications that appear when one restricts to the bounded case. The reader who is only interested in Eisenstein series can safely ignore the extra complexity.

\subsubsection{The construction}
Let $\mathcal{Y} \to \mathcal{Z}$ be a map of prestacks. Then, we define the formal completion
\[
	\mathcal{Z}^\wedge_\mathcal{Y} \defeq \mathcal{Y}_{\dR} \times_{\mathcal{Z}_{\dR}} \mathcal{Z}.
\]
Here, for any prestack $\mathcal{X}$, the de Rham prestack $\mathcal{X}_{\dR}$ of $\mathcal{X}$, is defined by the following functor of points
\[
	\Spec R \in \Aff^{\opp} \quad\mapsto\quad \mathcal{X}_{\dR}(\Spec R) = \mathcal{X}(\Spec \Ho^0(R)^{\red}),
\]
where $\Ho^0(R)^{\red}$ is the reduced ring associated to $\Ho^0(R)$.

When $\mathcal{Y}, \mathcal{Z} \in \Stk_{\perf, \lfp}$ with $\mathcal{Y}$ bounded, then
\[
	\IndCoh_0(\mathcal{Z}^\wedge_\mathcal{Y}) \defeq \IndCoh_0(\mathcal{Y} \to \mathcal{Z})
\]
is the full subcategory of $\IndCoh(\mathcal{Z}^\wedge_\mathcal{Y})$ that fits into the following Cartesian square
\[
\begin{tikzcd}
	\IndCoh_0(\mathcal{Z}^\wedge_\mathcal{Y}) \ar{r} \ar[hookrightarrow]{d} & \QCoh(\mathcal{Y}) \ar[hookrightarrow]{d}{\Upsilon} \\
	\IndCoh(\mathcal{Z}^\wedge_\mathcal{Y}) \ar{r} & \IndCoh(\mathcal{Y})
\end{tikzcd}
\]

When $\mathcal{Y}, \mathcal{Z} \in \Stk_{\lfp}$ and $\mathcal{Y}$ is perfect, the definition needs to be modified; see~\cite{beraldo_center_2020}*{Definition 4.1.6}. 

\subsubsection{Special cases} \label{subsubsec:IndCoh_0_special_niliso}
When $\mathcal{Y}$ is smooth, then $\QCoh(\mathcal{Y}) \simeq \IndCoh(\mathcal{Y})$ and hence, 
\[
	\IndCoh_0(\mathcal{Z}^\wedge_\mathcal{Y}) \simeq \IndCoh(\mathcal{Z}^\wedge_\mathcal{Y}).
\]

If $\mathcal{Y} \to \mathcal{Z}$ is a \emph{nil-isomorphism}, that is, the induced map $\mathcal{Y}_{\dR} \to \mathcal{Z}_{\dR}$ is an isomorphism, then
\[
	\mathcal{Z}^{\wedge}_\mathcal{Y} = \mathcal{Y}_{\dR} \times_{\mathcal{Z}_{\dR}} \mathcal{Z} \simeq \mathcal{Z}.
\]
Thus, if $\mathcal{Y}$ is smooth and $\mathcal{Y} \to \mathcal{Z}$ is a nil-isomorphism, then
\[
	\IndCoh_0(\mathcal{Z}^{\wedge}_\mathcal{Y}) \simeq \IndCoh(\mathcal{Z}^\wedge_\mathcal{Y}) \simeq \IndCoh(\mathcal{Z}).
\]

\subsubsection{Functoriality} \label{subsubsec:functoriality_IndCoh_0}
We will now recall functoriality of the assignment
\[
	(\mathcal{Y} \to \mathcal{Z}) \mapsto \IndCoh_0(\mathcal{Z}^\wedge_\mathcal{Y}). \teq \label{eq:indcoh_0_assignment}
\]

For any category $\mathcal{C}$, we use $\mathcal{C}^{\Delta^1} \defeq \Fun(\Delta^1, \mathcal{C})$ to denote the category of arrows in $\mathcal{C}$. Namely, objects are of the form $c_1 \to c_2$ and morphisms are the obvious commutative squares. When confusion is unlikely to occur, we will suppress the map from the notation and use $(c_1, c_2)$ to denote an object in $\mathcal{C}^{\Delta^1}$.

Consider the $1$-full subcategory $\Corr' (\mathcal{C}^{\Delta^1}) \subseteq \Corr(\mathcal{C}^{\Delta^1})$ which consists of all objects of $\Corr(\mathcal{C}^{\Delta^1})$ but which morphisms are given by correspondences of the form
\[
\begin{tikzcd}
	c_1 \ar{d} & c \ar[l] \ar{d} \ar{r}{\simeq} \arrow[phantom]{dl}[very near start]{\urcorner} & c_2 \ar{d} \\
	d_1 & \ar[l] d \ar{r} & d_2
\end{tikzcd}
\]
where, as indicated, the left square is Cartesian and the map $c\to c_2$ is an equivalence. It is easy to see that $\Corr'(\mathcal{C}^{\Delta^1})$ is closed under compositions of morphisms in $\Corr(\mathcal{C}^{\Delta^1})$, which are given by fiber products.

By~\cite{beraldo_center_2020}*{Proposition 4.7.2}, the assignment~\cref{eq:indcoh_0_assignment} upgrades to a functor\footnote{In fact, slightly more is true, and when boundedness condition is added, even more is true~\cite{beraldo_center_2020}*{Theorem 3.3.3}. However, this is all that we will need.}
\[
	\IndCoh_0: \Corr'((\Stk_{\perf, \lfp})^{\Delta^1}) \to \DGCat.
\]

\subsubsection{} \label{subsubsec:IndCoh_0_functoriality_explicit}
In particular, given a Cartesian diagram in $\Stk_{\perf, \lfp}$
\[
\begin{tikzcd}
	\mathcal{Y}_1 \ar{d} \ar{r} & \mathcal{Y}_2 \ar{d} \\
	\mathcal{Z}_1 \ar{r} & \mathcal{Z}_2
\end{tikzcd}
\]
we obtain a $!$-pullback functor $\IndCoh_0((\mathcal{Z}_2)^\wedge_{\mathcal{Y}_2}) \to \IndCoh_0((\mathcal{Z}_1)^\wedge_{\mathcal{Y}_1})$. Similarly, a diagram
\[
\begin{tikzcd}
	\mathcal{Y} \ar{d} \ar{r}{\simeq} & \mathcal{Y} \ar{d} \\
	\mathcal{Z}_1 \ar{r} & \mathcal{Z}_2
\end{tikzcd}
\]
in $\Stk_{\perf, \lfp}$ induces a $*$-pushforward functor $\IndCoh_0((\mathcal{Z}_1)^\wedge_{\mathcal{Y}}) \to \IndCoh_0((\mathcal{Z}_2)^\wedge_\mathcal{Y})$.

\subsubsection{Descent}
One salient feature of $\IndCoh_0$ is that it satisfies a strong form of descent. 

\begin{prop}[{\cite{beraldo_center_2020}*{Proposition 4.4.1}}] \label{prop:descent_for_IndCoh_0}
For any $\mathcal{W} \in \Stk_{\perf, \lfp}$, the contravariant functor $\IndCoh_0((-)^\wedge_\mathcal{W})$, via $?$-pullbacks, satisfies descent along any map $(\Stk_{\perf, \lfp})_{\mathcal{W}/}$.
\end{prop}

\begin{rmk}
In the bounded case, the $?$-pullback is inherited from the $!$-pullback of $\IndCoh$, which coincides with the $!$-pullback of $\IndCoh_0$ discussed above. In general, it is defined to be the right adjoint to the $*$-pushforward functor discussed above, see~\cite{beraldo_center_2020}*{\S4.2.2}.
\end{rmk}

Let us spell out what this means. Consider morphisms of stacks $\mathcal{W} \to \mathcal{X} \to \mathcal{Y}$ in $\Stk_{\perf, \lfp}$. The \Cech{} construction and the $?$-pullback functoriality of $\IndCoh_0$ gives an augmented co-simplicial object in $\DGCat$
\[
	\IndCoh_0(\mathcal{Y}^\wedge_\mathcal{W}) \to \IndCoh_0((\mathcal{X}^{\times_\mathcal{Y} (\bullet + 1)})^\wedge_\mathcal{W}).
\]
The proposition above asserts that this induces an equivalence of categories
\[
	\IndCoh_0(\mathcal{Y}^\wedge_\mathcal{W}) \simeq \Tot(\IndCoh_0((\mathcal{X}^{\times_\mathcal{Y} (\bullet + 1)})^\wedge_\mathcal{W})). \teq\label{eq:descent_IndCoh_0}
\]

\subsubsection{}
We will use a ``dual'' version of this proposition. Recall the following result from~\cite{lurie_higher_2017-1}*{Corollary 5.5.3.4} (see also~\cite{gaitsgory_study_2017}*{Volume 1, Chapter 1, Proposition 2.5.7}). Suppose we have a diagram $\Phi: I \to \DGCat$ such that all maps admit left adjoints, i.e. we obtain a diagram $\Phi^L: I^{\opp} \to \DGCat$. Then, we have an equivalence of categories
\[
	\lim_{i\in I} \Phi(i) \simeq \colim_{i\in I^{\opp}} \Phi^L(i).
\]

Applying this to \cref{eq:descent_IndCoh_0}, we obtain the following statement.

\begin{cor} \label{cor:covariant_descent_for_IndCoh_0}
For any string of morphisms of stacks $\mathcal{W} \to \mathcal{X} \to \mathcal{Y}$ in $\Stk_{\perf, \lfp}$, we have a natural equivalence of categories
\[
	\IndCoh_0(\mathcal{Y}^\wedge_\mathcal{W}) \simeq |\IndCoh_0((\mathcal{X}^{\times_\mathcal{Y} (\bullet + 1)})^\wedge_\mathcal{W})|,
\]
where $|-|$ denotes geometric realization, i.e. colimit of a simplicial category. Moreover, the functors used in the simplicial structure are the $*$-pushforward functors.
\end{cor}

\subsubsection{Relative tensor over $\QCoh$}
The formation of $\IndCoh_0$ behaves nicely with respect to tensoring over $\QCoh$.

\begin{prop} \label{prop:categorical_Kuenneth}
Consider a diagram $\mathcal{U} \to \mathcal{V} \to \mathcal{Z} \leftarrow \mathcal{Y} \leftarrow \mathcal{X}$ in $\Stk_{\perf, \lfp}$. Then, the exterior product descends to an equivalence
\[
	\IndCoh_0(\mathcal{V}^\wedge_{\mathcal{U}}) \otimes_{\QCoh(\mathcal{Z})} \IndCoh_0(\mathcal{Y}^\wedge_{\mathcal{X}}) \xrightarrow{\simeq} \IndCoh_0((\mathcal{V}\times_\mathcal{Z} \mathcal{Y})^\wedge_{\mathcal{U}\times_\mathcal{Z} \mathcal{X}}). \teq\label{eq:categorical_Kuenneth}
\]
\end{prop}
\begin{proof}
The LHS of \cref{eq:categorical_Kuenneth} can be rewritten as
\begin{align*}
	(\IndCoh&_0(\mathcal{V}^\wedge_{\mathcal{U}}) \otimes \IndCoh_0(\mathcal{Y}^\wedge_\mathcal{X})) \otimes_{\QCoh(\mathcal{Z})\otimes \QCoh(\mathcal{Z})} \QCoh(\mathcal{Z}) \\
	&\simeq \IndCoh_0((\mathcal{V} \times \mathcal{Y})^\wedge_{\mathcal{U} \times \mathcal{X}}) \otimes_{\QCoh(\mathcal{Z}\times \mathcal{Z})} \QCoh(\mathcal{Z}),
\end{align*}
where we used~\cite{beraldo_center_2020}*{Proposition 4.5.5} and~\cite{ben-zvi_integral_2010}*{Theorem 4.7}. But now, this is equivalent to the RHS of~\cref{eq:categorical_Kuenneth}, again by~\cite{beraldo_center_2020}*{Proposition 4.5.5}.
\end{proof}

\begin{rmk}
As written, \cite{beraldo_center_2020}*{Proposition 4.5.5} requires that the stacks involved are bounded. However, the proof does not use this fact. In fact, this result, without the bounded condition, is used in \cite{beraldo_center_2020}*{\S4.6.1}.

D. Beraldo communicated to us that even even perfectness could be relaxed. In fact, this is implicitly used in~\cite{beraldo_topological_2019}.
\end{rmk}

\subsection{Factorization homology} \label{subsec:prelim_fact_hom}

As mentioned above, factorization is the tool we use to formulate manifold gluing for spectral Eisenstein series. Roughly speaking, for each $\En_n$-algebra, factorization homology is a homology theory of $n$-dimensional manifolds which satisfy a multiplicative form of excision. We will now give a brief overview of the theory. The reader is referred to~\cite{ayala_factorization_2015} for a detailed treatment. 

\subsubsection{$\En_n$-algebras}
Let $\Disk_n$ denote the symmetric monoidal $\infty$-category whose objects are non-empty finite disjoint unions of $n$-dimensional disks and whose morphisms are given by the space of open embeddings with the compact-open topology. Let $\mathcal{V}$ be a $\otimes$-presentable symmetric monoidal category in the sense of~\cite{ayala_factorization_2015}*{Definition 3.4}. Namely, $\mathcal{V}$ is a presentable symmetric monoidal category such that $\otimes$ is continuous in each variable, i.e. it commutes with colimits in each variable. An $\En_n$-algebra $\mathsf{A}$ in $\mathcal{V}$ is, by definition, a symmetric monoidal functor $\mathsf{A}: \Disk_n \to \mathcal{V}$. Namely, the category of $\En_n$-algebras in $\mathcal{V}$ is given by\footnote{Strictly speaking, what we define here is the category of framed $\En_n$-algebras. See also~\cite{ayala_factorization_2015}*{Remarks after Example 2.11}.}
\[
	\EnAlg{n}(\mathcal{V}) \defeq \Fun^{\otimes}(\Disk_n, \mathcal{V}).
\]

Given an $\En_n$-algebra $\mathsf{A}$, we will usually write $\mathsf{A}$ to also denote its value on a single disk $D^n$.

\subsubsection{Manifolds}
We denote by $\Mnfd_n$ the symmetric monoidal category of $n$-dimensional manifolds which admit a good cover. Moreover, morphisms are given by the spaces of open embeddings. By a \emph{good cover}, we mean a finite cover by Euclidean spaces with the property that each non-empty intersection is itself homeomorphic to an Euclidean space. 

Note that any manifold admitting a good cover has the homotopy type of a finite CW-complex. Moreover, it is clear that $\Disk_n$ is a full subcategory of $\Mnfd_n$.

\subsubsection{Factorization homology}
Let $\mathsf{A} \in \EnAlg{n}(\mathcal{V}) = \Fun^{\otimes}(\Disk_n, \mathcal{V})$ be an $\En_n$-algebra in $\mathcal{V}$. Factorization homology is defined as the left Kan extension of $\mathsf{A}$ along the fully faithful embedding $\Disk_n \hookrightarrow \Mnfd_n$.

More concretely, for each $M \in \Mnfd_n$, we consider
\[
	(\Disk_n)_{/M} \defeq \Disk_n \times_{\Mnfd_n} (\Mnfd_{n})_{/M}.
\]
The factorization homology of $M$ with coefficients in $\mathsf{A}$ is given by
\[
	\int_M \mathsf{A} \defeq \colim_{(D^{\sqcup k} \hookrightarrow M) \in (\Disk_n)_{/M}} A(D)^{\otimes k}. \teq\label{eq:fact_hom_defn}	
\]

\subsubsection{$\otimes$-excision}
In practice, however, we usually don't use \cref{eq:fact_hom_defn} to compute factorization homology. Rather, we use the fact that factorization homology satisfies, and in fact, is characterized by, a multiplicative version of excision, which we will now recall.

\begin{defn} \label{defn:collar_gluing}
A \emph{collar gluing} of manifolds is a continuous map $f: M \to [-1, 1]$, such that the restriction $f|_{(-1, 1)}: M|_{(-1, 1)} \to (-1, 1)$ is a manifold bundle. We denote a collar gluing as $M_1 \cup_{M_0 \times \mathbb{R}} M_2 = M$, where $M_1 = f^{-1}([-1, 1))$, $M_2 = f^{-1}((-1, 1])$, and $M_0 = f^{-1}(0)$.
\end{defn}

\begin{defn} \label{defn:homology_theory}
A \emph{homology theory} for $n$-manifolds valued in $\mathcal{V}$ is a symmetric monoidal functor $\mathsf{E}: \Mnfd_n \to \mathcal{V}$ which satisfies $\otimes$-excision, i.e. for any collar gluing $M = M_1 \cup_{M_0\times \mathbb{R}} M_2$, the natural map
\[
	\mathsf{E}(M_1) \otimes_{\mathsf{E}(M_0 \times \mathbb{R})} \mathsf{E}(M_2) \to \mathsf{E}(M)
\]
is an equivalence. We denote by $\HH(\Mnfd_n, \mathcal{V})$ the full subcategory of $\Fun^{\otimes}(\Mnfd_n, \mathcal{V})$ spanned by functors which satisfy $\otimes$-excision.
\end{defn}

\begin{thm}[\cite{ayala_factorization_2015}] \label{thm:otimes_excision}
Let $\mathcal{V}$ be a symmetric monoidal $\infty$-category which is $\otimes$-presentable. Then, there are mutually inverse functors
\[
	\int: \EnAlg{n}(\mathcal{V}) \rightleftarrows \HH(\Mnfd_n, \mathcal{V}): \ev_{\Disk_n}
\]
between the category of $\En_n$-algebras in $\mathcal{V}$ and the category of homology theories valued in $\mathcal{V}$. Here $\ev_{\Disk_n}$ and $\int$ are given by restricting to and, respectively, left Kan extension along $\Disk_n \hookrightarrow \Mnfd_n$. The latter is, by definition, the functor of taking factorization homology.
\end{thm}

\subsubsection{Compact manifolds with boundary} \label{subsubsec:boundary_of_manifolds}
Even though our manifolds are, technically speaking, without boundary, we will make use of their ``boundary'' in our construction. We will now explain what this means.

Let $\Mnfd'_n$ be the category of compact $n$-manifolds with possibly non-empty boundary $\partial M$ such that its interior $\interior{M} = M \setminus \partial M$ admits a good cover. Moreover, morphisms in $\Mnfd'_n$ are given by (necessarily closed) embeddings. Taking the interior gives a natural functor of $\infty$-categories
\[
	F: \Mnfd'_n \to \Mnfd_n, \qquad M \mapsto \interior{M} = M \setminus \partial M.
\]

\begin{lem} \label{lem:open_manifold_vs_manifold_w_boundary}
The functor $F$ is an equivalence of categories. We write $M\mapsto \lbar{M}$ to denote an inverse of $F$.
\end{lem}
\begin{proof}
By the remarks after~\cite{ayala_factorization_2015}*{Definition 2.1}, each manifold $M\in \Mnfd_n$ is the interior of a compact manifold $\lbar{M}$. In other words, the functor $F$ is essentially surjective. It remains to show that $F$ is also fully faithful. Namely, for $M, N \in \Mnfd'_n$, we want to show that the following map is an equivalence
\[
	F_{M, N}: \Emb(M, N) \to \Emb(\interior{M}, \interior{N}).
\]

By the existence of collar neighborhoods,~\cites{brown_locally_1962,connelly_new_1971},\footnote{See also~\cite{baillif_collared_2020}*{Theorem 1.2} for a summary of results regarding the existence of collar neighborhoods.} for any $M\in \Mnfd'_n$, we have a pair of embeddings
\[
\begin{tikzcd}
	\interior{M} \ar[hookrightarrow, shift left=\arrdisp]{r}{\iota_M} & \ar[hookrightarrow, shift left=\arrdisp]{l}{\tilde{\iota}_M} M
\end{tikzcd}
\]
where $\iota_M$ is the canonical embedding, such that the compositions in both ways are isotopically equivalent to the identity maps. These maps induce the following pair of morphisms
\[
\begin{tikzcd}
	\Emb(M, N) \ar[shift left=\arrdisp]{r}{\iota_M^*} & \ar[shift left=\arrdisp]{l}{\tilde{\iota}_M^*} \Emb(\interior{M}, N)
\end{tikzcd}
\]
such that the compositions in both ways are homotopy equivalent to the identity maps. In particular, they are both homotopy equivalences. Similarly, we have
\[
\begin{tikzcd}
	\Emb(\interior{M}, N) \ar[shift left=\arrdisp]{r}{\tilde{\iota}_{N, *}} & \ar[shift left=\arrdisp]{l}{\iota_{N, *}} \Emb(\interior{M}, \interior{N})
\end{tikzcd}
\]
that are mutually inverse homotopy equivalences.

By construction, $\iota_M^*$ factors as follows
\[
\begin{tikzcd}
	\Emb(M, N) \ar{dr}{\simeq}[swap]{\iota_M^*} \ar{r}{F_{M, N}} & \Emb(\interior{M}, \interior{N}) \ar{d}{\iota_{N, *}}[swap]{\simeq} \\
	& \Emb(\interior{M}, N)
\end{tikzcd}
\]
Thus, $F_{M, N}$ is also a homotopy equivalence and we are done.
\end{proof}

\subsubsection{} \label{subsubsec:abuse_notation_boundary} Because of the equivalence stated in \cref{lem:open_manifold_vs_manifold_w_boundary}, throughout this paper, we will not make a distinction between manifolds without boundaries and compact manifolds with (possibly non-empty) boundaries, unless confusion is likely to occur. For instance, when $M\in \Mnfd_n$, by abuse of notation, we use 
\[
	\partial M \defeq \partial \lbar{M} \defeq \lbar{M} \setminus M
\]
to denote the boundary of $\lbar{M}$.

Let $\Disk_n' = \Disk_n \times_{\Mnfd_n} \Mnfd_n'$. We obtain an equivalence of symmetric monoidal categories $\Disk'_n \simeq \Disk_n$. Thus, an $\En_n$-algebra is, equivalently, a symmetric monoidal functor out of $\Disk_n'$.

\section{$\En_n$-Hecke categories and Eisenstein series}
In this section, we will construct the $\En_n$-Hecke category and compute its factorization homology on topological manifolds. More precisely, we start, in~\cref{subsec:construction_of_Eis}, with the construction of the functor $\Eis$, which is a symmetric monoidal functor out of $\Mnfd_n$. Its restriction to $\Disk_n$ gives an $\En_n$-category, the so-called $\En_n$-Hecke category, which is discussed in~\cref{subsec:En-Hecke}. The main theorem, which says that $\Eis$ is a homology theory in the sense of \cref{defn:homology_theory}, is stated and proved in~\cref{subsec:main_theorem}. Note that this is equivalent to saying that the value of $\Eis$ on an $n$-dimensional manifold $M$ is the factorization homology of the associated $\En_n$-Hecke category over $M$. 

In~\cref{subsec:BP_BG}, we specialize to the case of $\En_2$-Hecke categories $\Hecke_{G, P}$ as appearing in the Langlands program; the important point is the Hecke pair condition, which makes everything more explicit. And finally, in~\cref{subsec:non-compact_case}, we show that for non-compact manifolds, Eisenstein series can be defined by a simple set-theoretic support condition (as opposed to the appearance of the more complicated $\IndCoh_0$).

\subsection{The functor $\Eis$} \label{subsec:construction_of_Eis}
In this subsection, we will give the constructions of the main objects of this paper: Eisenstein series and $\En_n$-Hecke categories. Throughout we will fix a pair of stacks $\mathcal{Y} \to \mathcal{Z}$ such that both $\mathcal{Y}$ and $\mathcal{Z}$ are perfect and locally of finite presentation.
	
\begin{rmk} \label{rmk:perfectness_lfp_mapping_stacks}
These two conditions behave nicely with respect to forming derived mapping stacks. Indeed, fix a finite CW complex $M$ (of arbitrary dimension). Suppose $\mathcal{Y}$ is perfect, then so is $\mathcal{Y}^M$ by~\cite{ben-zvi_integral_2010}*{Corollary 3.25}. Moreover, if we are given $\mathcal{Y} \to \mathcal{Z}$ with $\mathcal{Y}$ and $\mathcal{Z}$ being perfect, then so is $\mathcal{Y}^N \times_{\mathcal{Z}^N} \mathcal{Z}^M$ for any map of finite CW complexes $N\to M$. Indeed, this is because being perfect is closed under fiber products~\cite{ben-zvi_integral_2010}*{Proposition 3.24}. 

The discussion above also applies when replacing the perfect condition with being locally of finite presentation via a simple cotangent complex calculation. 
\end{rmk}

\subsubsection{Eisenstein homology theory}
The goal is to construct a homology theory for $n$-dimensional manifolds in the sense of \cref{defn:homology_theory}
\[
	\Eis_n(\mathcal{Y}, \mathcal{Z}): \Mnfd_n \to \DGCat
\]
as a composition of three different functors
\[
\begin{tikzcd}
	\Mnfd_n \ar[bend right=12]{rrr}[swap]{\Eis_n(\mathcal{Y}, \mathcal{Z})} \ar{r}{B} & \Corr'((\Spc^{\Delta^1}_{\fin})^{\opp}) \ar{r}{\cMap} & \Corr'((\Stk_{\perf, \lfp})^{\Delta^1}) \ar{r}{\IndCoh_0} & \DGCat 
\end{tikzcd} \teq\label{eq:composition_Eis}
\]
We will now construct this functor as a symmetric monoidal functor. The proof that its a homology theory (which is equivalent to our main result, \cref{thm:main}) will be carried out in~\cref{subsec:main_theorem}.

\subsubsection{}
We start with a variant of the category of correspondences described in \cref{subsubsec:functoriality_IndCoh_0}. For any category $\mathcal{C}$, consider the $1$-full subcategory $\Corr'((\mathcal{C}^{\Delta^1})^{\opp})$ of $\Corr((\mathcal{C}^{\Delta^1})^{\opp})$ which consists of all objects but which morphisms are given by cospans of the form
\[
\begin{tikzcd}
	c_1 \ar{d} \ar{r} & c \ar{d} & c_2 \ar{d} \ar{l} \\
	d_1 \ar{r} & d \arrow[phantom]{ul}[very near start]{\lrcorner} & d_2 \ar{l}[swap]{\simeq}
\end{tikzcd}
\]
where, as indicated, the left square is a pushout and the map $d_2 \to d$ is an equivalence.

When $\mathcal{C} = \Spc_{\fin}$, the category $\Corr'((\mathcal{C}^{\Delta^1})^{\opp}) = \Corr'((\Spc^{\Delta^1}_{\fin})^{\opp})$ is the target of the functor $B$ in~\cref{eq:composition_Eis}, which will be now described.\footnote{$B$ stands for \emph{Boundary}.} At the level of objects (see~\cref{subsubsec:boundary_of_manifolds} for the notation),
\[
	B(M) = (\partial M \to \lbar{M}) \in \Spc_{\fin}^{\Delta^1}.
\]
Moreover, $B$ sends an open embedding $N \hookrightarrow M$ to the morphism in $\Corr'((\Spc^{\Delta^1}_{\fin})^{\opp})$ is given by the following diagram
\[
\begin{tikzcd}
	\partial N \ar{d} \ar{r} & \lbar{M} \setminus N \ar{d} & \partial M \ar{d} \ar{l} \\
	\lbar{N} \ar{r} & \lbar{M} \arrow[phantom]{ul}[very near start]{\lrcorner} & \lbar{M} \ar{l}[swap]{\simeq}
\end{tikzcd} \teq\label{eq:functoriality_defects}
\]

\subsubsection{} \label{subsubsec:cMap_defn}
We now turn to the functor $\cMap$ of~\cref{eq:composition_Eis}. We have a natural functor
\[
	\cMap: (\Spc^{\Delta^1})^{\opp} \to \Stk^{\Delta^1}
\]
which assigns to each object $(N \to M) \in (\Spc^{\Delta^1})^{\opp}$ an object $(\mathcal{Y}^M \to (\mathcal{Y}, \mathcal{Z})^{N, M}) \in \Stk^{\Delta^1}$. Here, $(\mathcal{Y}, \mathcal{Z})^{N, M} \defeq \mathcal{Y}^N \times_{\mathcal{Z}^N} \mathcal{Z}^M$ is precisely the stack of commutative squares
\[
\begin{tikzcd}
	N \ar{r} \ar{d} & \mathcal{Y} \ar{d} \\
	M \ar{r} & \mathcal{Z}
\end{tikzcd}
\]
Moreover, it is easy to see that this functor automatically upgrades to a functor
\[
	\cMap: \Corr'((\Spc^{\Delta^1})^{\opp}) \to \Corr'(\Stk^{\Delta^1}),
\]
and hence, the functor $\cMap$ of~\cref{eq:composition_Eis}. Here, the extra conditions such as $\perf$ and $\lfp$ on stacks are guaranteed to hold by \cref{rmk:perfectness_lfp_mapping_stacks}.

By construction, we know that $\cMap$ turns colimits in $\Spc^{\Delta^1}$ to limits in $\Stk^{\Delta^1}$.

\subsubsection{}
Finally, the functor $\IndCoh_0$ of~\cref{eq:composition_Eis} is given by~\cref{subsubsec:functoriality_IndCoh_0}.

\subsubsection{}
It is easy to see that $B$ and $\cMap$ are symmetric monoidal. Moreover, $\IndCoh_0$ is also symmetric monoidal, by \cref{prop:categorical_Kuenneth}. Thus, $\Eis_n = \IndCoh_0 \circ \cMap \circ B$ is symmetric monoidal.

\subsubsection{} For future reference, we note that for $M\in \Mnfd_n$, we have
\[
	\Eis_n(\mathcal{Y}, \mathcal{Z})(M) = \IndCoh_0\left(\left((\mathcal{Y}, \mathcal{Z})^{\partial M, \lbar{M}}\right)^{\wedge}_{\mathcal{Y}^{\lbar{M}}}\right) \simeq \IndCoh_0\left(\left((\mathcal{Y}, \mathcal{Z})^{\partial M, \lbar{M}}\right)^{\wedge}_{\mathcal{Y}^{\lbar{M}}}\right).
\]
Moreover, by abuse of notation (see also~\cref{subsubsec:abuse_notation_boundary}), we will sometimes write $(\partial M, M)$ in place of $(\partial \lbar{M}, \lbar{M})$. The above thus becomes
\[
	\Eis_n(\mathcal{Y}, \mathcal{Z})(M) = \IndCoh_0\left(\left((\mathcal{Y}, \mathcal{Z})^{\partial M, M}\right)^{\wedge}_{\mathcal{Y}^M}\right).
\]

\subsection{$\En_n$-Hecke categories}
\label{subsec:En-Hecke}
Let $\mathcal{Y} \to \mathcal{Z}$ be as above. Let $\Hecke_n(\mathcal{Y}, \mathcal{Z}): \Disk_n \to \DGCat$ be a symmetric monoidal functor obtained by restricting $\Eis_n(\mathcal{Y}, \mathcal{Z})$ along the fully faithful embedding $\Disk_n \hookrightarrow \Mnfd_n$. By definition, the value of $\Hecke_n(\mathcal{Y}, \mathcal{Z})$ on a $n$-dimensional disk $D^n$ is given by
\[
	\Hecke_n(\mathcal{Y}, \mathcal{Z})(D^n) = \IndCoh_0\left(\left((\mathcal{Y}, \mathcal{Z})^{S^{n-1}, D^n}\right)^\wedge_{\mathcal{Y}^{D^n}}\right) \simeq \IndCoh_0((\mathcal{Y}^{S^{n-1}} \times_{\mathcal{Z}^{S^{n-1}}} \mathcal{Z})^\wedge_{\mathcal{Y}}).
\]
As usual, we will use $\Hecke_n(\mathcal{Y}, \mathcal{Z})$ to denote $\Hecke_n(\mathcal{Y}, \mathcal{Z})(D^n)$.

For the reader's convenience, let us unwind the $\En_n$-monoidal structure. Note that the $\En_n$-monoidal structure is, by construction, induced by~\cref{eq:functoriality_defects}. Indeed, for each open embedding $\iota: (D^n)^{\sqcup k} \to D^n$, we have the following cospan in $\Spc^{\Delta^1}$
\[
\begin{tikzcd}
	(S^{n-1})^{\sqcup k} \ar{d} \ar{r} & D^n \setminus \iota((D^n)^{\sqcup k}) \ar{d} & S^{n-1} \ar{d} \ar{l} \\
	(D^n)^{\sqcup k} \ar{r}{\iota} & D^n \arrow[phantom]{ul}[very near start]{\lrcorner} & D^n \ar{l}[swap]{\simeq}
\end{tikzcd} \teq\label{eq:pair_of_pants}
\]
Now, as in~\cref{eq:composition_Eis}, applying $\cMap$, we obtain a correspondence, whose $\IndCoh_0$-pull-and-push gives the desired $\En_n$-multiplication structure. See \cref{fig:pair_of_pants} for an illustration.

\begin{figure}[htb]
\centering
\includegraphics[height=1in]{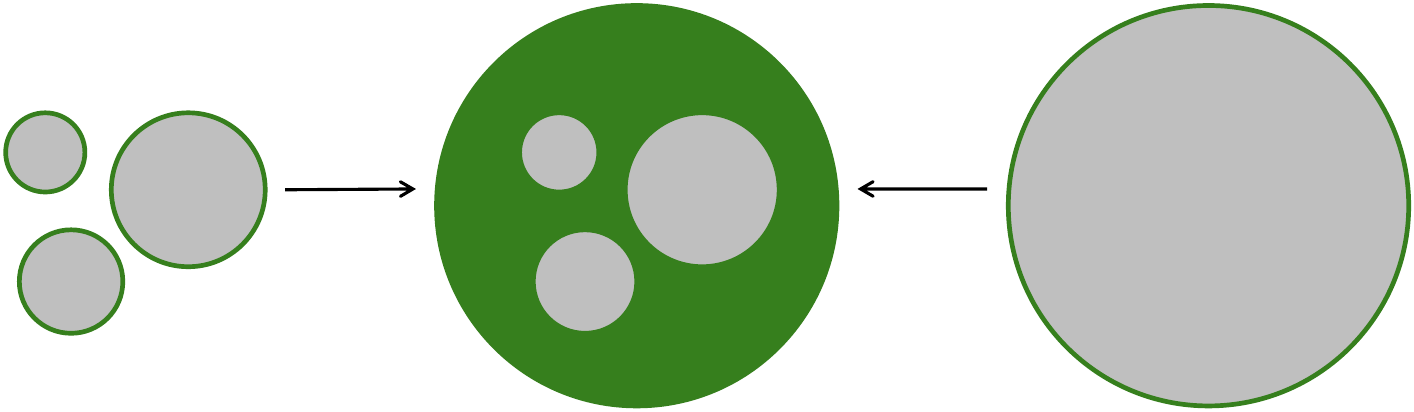}
\caption{$\En_n$-structure from the pair of pants construction.} \label{fig:pair_of_pants}
\begin{descriptionBox}
This figure illustrates~\cref{eq:pair_of_pants} when $k=3$ and $n=2$. Here, we collapse the columns of~\cref{eq:pair_of_pants}. The green parts represent the objects in top row of~\cref{eq:pair_of_pants} whereas the gray parts represent the parts of the bottom row that are not in the images of the top row.
\end{descriptionBox}
\end{figure}

\subsection{The main theorem} \label{subsec:main_theorem}
By the definition of factorization homology as a left Kan extension, for any $n$-dimensional manifold $M$, we have a natural map
\[
	\int_M \Hecke_n(\mathcal{Y}, \mathcal{Z}) \to \Eis_n(\mathcal{Y}, \mathcal{Z})(M).
\]
The rest of this subsection will be dedicated to the proof of our main theorem, which states that this map is an equivalence.

\begin{thm} \label{thm:main}
Let $\mathcal{Y} \to \mathcal{Z}$ be a morphism between stacks that are perfect and locally of finite presentation. Then, we have a natural equivalence
\[
	\int_M \Hecke_n(\mathcal{Y}, \mathcal{Z}) \simeq \Eis_n(\mathcal{Y}, \mathcal{Z})(M) = \IndCoh_0\left(\left((\mathcal{Y}, \mathcal{Z})^{\partial M, M}\right)^{\wedge}_{\mathcal{Y}^M}\right)
\]
for any topological manifold $M \in \Mnfd_n$.
\end{thm}

For the remainder of \cref{subsec:main_theorem}, to keep the notation less cluttered, we will write $\Hecke$ and $\Eis$ in place of $\Hecke_n(\mathcal{Y}, \mathcal{Z})$ and $\Eis_n(\mathcal{Y}, \mathcal{Z})$ respectively, with $\mathcal{Y} \to \mathcal{Z}$ a fixed morphism of stacks with $\mathcal{Y}$ and $\mathcal{Z}$ being perfect and locally of finite presentation.

\subsubsection{Homology theory}
By \cref{thm:otimes_excision}, to prove \cref{thm:main}, it suffices to show that $\Eis$ is a homology theory, i.e. that it satisfies $\otimes$-excision. More explicitly, let $M = M_1 \cup_{M_0 \times \mathbb{R}} M_2$ be a collar gluing in the sense of \cref{defn:collar_gluing}. We want to show that the following natural map is an equivalence
\[
	\Eis(M_1) \otimes_{\Eis(M_0 \times \mathbb{R})} \Eis(M_2) \xrightarrow{\simeq} \Eis(M). \teq\label{eq:Eis_satisfies_otimes_excision}
\]

The algebra (i.e. $\En_1$-monoidal) structure of $\Eis(M_0 \times \mathbb{R})$ as well as the module structures of $\Eis(M_1)$ and $\Eis(M_2)$ over it are induced by ``cylinder stacking.'' Unwinding the definition, we see that these structures are obtained via pulling and pushing through a correspondence induced by diagrams of the form \cref{eq:functoriality_defects}. 

In what follows, we will work relatively over a symmetric monoidal category $\mathcal{B}$. This effectively ``absorbs'' the first square of \cref{eq:functoriality_defects} so that the algebra and module structures only involve the second square of \cref{eq:functoriality_defects}. In terms of sheaves, this means that our structures only involve pushforward rather than both pushforward and pullback.\footnote{This technique has been used in many places to overcome similar technical difficulties, for example \cites{ben-zvi_character_2009,ben-zvi_betti_2021,beraldo_topological_2019}.}

\subsubsection{Working relatively}
Let $\mathcal{B} = \IndCoh_0((\mathcal{Y}^{M_0\times \mathbb{R}})^\wedge_{\mathcal{Y}^{M_0\times \mathbb{R}}}) \simeq \QCoh(\mathcal{Y}^{M_0 \times \mathbb{R}}) \simeq \QCoh(\mathcal{Y}^{M_0})$ be equipped with the standard symmetric monoidal structure. A diagram chase shows that we have a monoidal functor $\QCoh(\mathcal{Y}^{M_0}) \to \Eis(M_0 \times \mathbb{R})$ given by $\IndCoh_0$ $*$-pushforward along (see~\cref{subsubsec:IndCoh_0_functoriality_explicit})
\[
\begin{tikzcd}
	\mathcal{Y}^{M_0} \ar{d} \ar{r} & \mathcal{Y}^{M_0} \ar{d} \\
	\mathcal{Y}^{M_0} \ar{r} & (\mathcal{Y}, \mathcal{Z})^{\partial(M_0 \times \mathbb{R}), M_0 \times \mathbb{R}}
\end{tikzcd}
\]

This induces right and left $\mathcal{B}$-module structures on $\Eis(M_0\times \mathbb{R})$, a right $\mathcal{B}$-module structure on $\Eis(M_1)$, and a left $\mathcal{B}$-module structure on $\Eis(M_2)$. For $W=M_1$, $M_0\times \mathbb{R}$ or $M_2$, the module structure on $\Eis(W)$ is canonically identified with the $\IndCoh_0$ $!$-pullback along $(\mathcal{Y}, \mathcal{Z})^{\partial W, W} \to (\mathcal{Y}, \mathcal{Z})^{\partial W, W} \times \mathcal{Y}^{M_0}$ induced by $(\partial W \sqcup M_0, W\sqcup M_0) \to (\partial W, W)$. Note that in the case where $W = M_0\times \mathbb{R}$, there are two possible inclusions $M_0\to \partial(M_0\times \mathbb{R}) = \partial W$, corresponding to the two module structures given by left and right multiplications.

\begin{figure}[htb]
\centering
\includegraphics[height=1in]{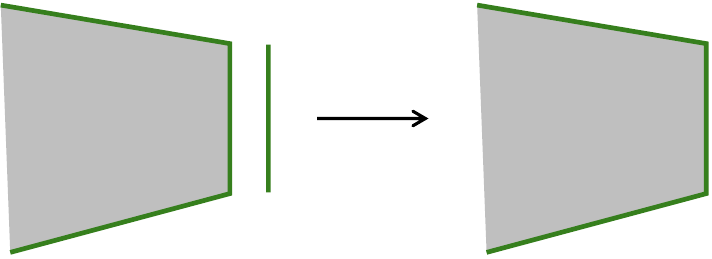}
\caption{$\mathcal{B}$-module structure on $\Eis(M_1)$.} \label{fig:module_structure}
\begin{descriptionBox}
This figure illustrates the map $(\partial M_1 \sqcup M_0, M_1 \sqcup M_0) \to (\partial M_1 , M_1)$ that induces the $\mathcal{B}$-module structure on $\Eis(M_1)$. Here, the green vertical line on the left represents $M_0$. The map sends it to the vertical segment of $\partial M_1$ on the right. 
\end{descriptionBox}
\end{figure}

\subsubsection{Relative bar complex}
We have the following augmented simplicial category
\[
	\Eis(M_1)\otimes_{\mathcal{B}} \Eis(M_0\times \mathbb{R})^{\otimes_{\mathcal{B}} \bullet} \otimes_{\mathcal{B}} \Eis(M_2) \to \Eis(M), \teq\label{eq:augmented_simplicial}
\]
where $\Eis(M)$ is the augmentation, i.e. it lives in simplicial degree $-1$. Moreover, the geometric realization of the LHS of~\cref{eq:augmented_simplicial} computes the LHS of~\cref{eq:Eis_satisfies_otimes_excision}. Namely, we have an equivalence
\[
	|\Eis(M_1)\otimes_{\mathcal{B}} \Eis(M_0\times \mathbb{R})^{\otimes_{\mathcal{B}} \bullet} \otimes_{\mathcal{B}} \Eis(M_2)| \simeq \Eis(M_1) \otimes_{\Eis(M_0\times \mathbb{R})} \Eis(M_2).
\]

The terms on the LHS of~\cref{eq:augmented_simplicial} can be easily computed. Indeed, by \cref{prop:categorical_Kuenneth}, we have
\begin{align*}
	\Eis(&M_1)\otimes_{\mathcal{B}} \Eis(M_0\times \mathbb{R})^{\otimes_{\mathcal{B}} k} \otimes_{\mathcal{B}} \Eis(M_2) \\
	&\simeq \IndCoh_0\left(\left((\mathcal{Y}, \mathcal{Z})^{\partial M_1 \sqcup_{M_0} \partial (M_0\times \mathbb{R})^{\sqcup_{M_0} k} \sqcup_{M_0} \partial M_2, M_1 \sqcup_{M_0} (M_0\times \mathbb{R})^{\sqcup_{M_0} k} \sqcup_{M_0} M_2}\right)^\wedge_{\mathcal{Y}^{M_1 \sqcup_{M_0} (M_0\times \mathbb{R})^{\sqcup_{M_0} k} \sqcup_{M_0} M_2}}\right) \\
	&\simeq \IndCoh_0\left(\left((\mathcal{Y}, \mathcal{Z})^{\partial M_1 \sqcup_{M_0} \partial (M_0\times \mathbb{R})^{\sqcup_{M_0} k} \sqcup_{M_0} \partial M_2, M}\right)^\wedge_{\mathcal{Y}^{M}}\right).
\end{align*}
Here, for the first equivalence, we also use the fact that the $(\mathcal{Y}, \mathcal{Z})^{-, -}$ construction (see~\cref{subsubsec:cMap_defn}), turns colimits to limits.

It remains to show that
\[
	\left|\IndCoh_0\left(\left((\mathcal{Y}, \mathcal{Z})^{\partial M_1 \sqcup_{M_0} \partial (M_0\times \mathbb{R})^{\sqcup_{M_0} \bullet} \sqcup_{M_0} \partial M_2, M}\right)^\wedge_{\mathcal{Y}^{M}}\right)\right| \to \IndCoh_0\left(\left((\mathcal{Y}, \mathcal{Z})^{\partial M, M}\right)^\wedge_{\mathcal{Y}^M}\right) \defeq \Eis(M) \teq\label{eq:final_equivalence}
\]
is an equivalence.

\subsubsection{An alternative description of the simplicial object} 
It is easy to see an alternative way to obtain the simplicial category on the far left of~\cref{eq:final_equivalence}. Indeed, consider the following morphism $\eta: (\partial M , M) \to (\partial M_1 \sqcup_{M_0} \partial M_2, M)$ in $\Spc^{\Delta^1}$ 
\[
\begin{tikzcd}
	\partial M_1 \sqcup_{M_0} \partial M_2 \ar{d} & \partial M \ar{l} \ar{d} \\
	M & M \ar{l}[swap]{\simeq}
\end{tikzcd}
\]
Let $\coCechNv^\bullet(\eta)$ be the \coCech{} nerve of this morphism, which is a co-simplicial object in $\Spc^{\Delta^1}_{\fin}$. Applying the $(\mathcal{Y}, \mathcal{Z})^{-, -}$ construction (see~\cref{subsubsec:cMap_defn}) and $\IndCoh_0$ (using the $\IndCoh_0$ $*$-pushforward), we obtain precisely the simplicial category appearing on the far left of~\cref{eq:final_equivalence}.

\begin{figure}[htb]
\centering
\includegraphics[height=0.9in]{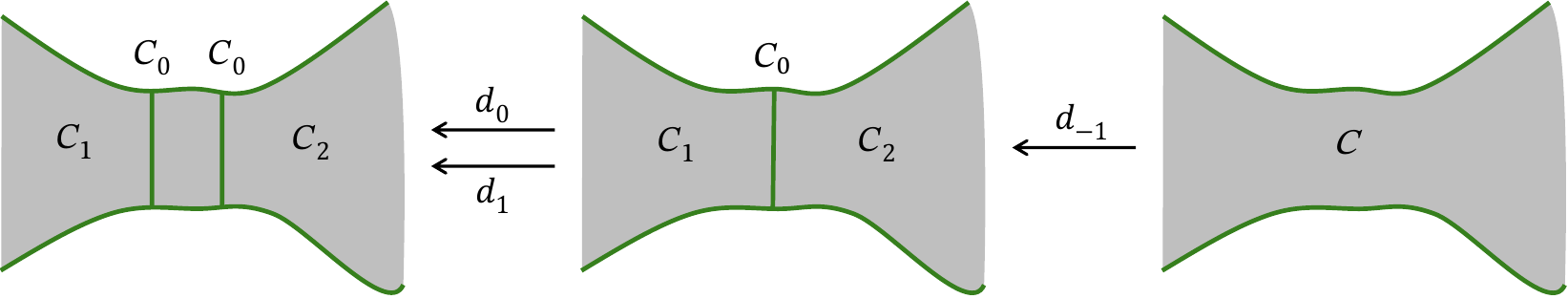}
\caption{\coCech{} resolution.} \label{fig:coCech}

\begin{descriptionBox}
This figure illustrates the first two steps of the \coCech{} nerve of the morphism $(\partial M , M) \to (\partial M_1 \sqcup_{M_0} \partial M_2, M)$ in $\Spc^{\Delta^1}$. The two items on the left represent the zero-th and first steps in the co-simplicial resolution whereas the item on the right is the co-augmentation. 
\end{descriptionBox}
\end{figure}

Now, by descent of $\IndCoh_0$, \cref{cor:covariant_descent_for_IndCoh_0}, we obtain that the morphism in~\cref{eq:final_equivalence} is an equivalence. This concludes the proof of the main theorem,~\cref{thm:main}.

\subsection{The case of $BP \to BG$} \label{subsec:BP_BG}
As mentioned in the introduction~\cref{subsubsec:Hecke_pair_condition}, the sheaf theory $\IndCoh_0$ is much simpler in the case of bounded stacks. In this subsection, we specialize to the case of $\En_2$-Hecke categories and formulate the Hecke pair condition. The Hecke pair condition is designed precisely to make sure that all stacks appearing in the proof of \cref{thm:main} are perfect, of finite presentation, and bounded. The crucial point, for us, is that the case $BP\to BG$ satisfies the Hecke pair condition.

\subsubsection{Hecke pair}
We start with the definition of a Hecke pair.

\begin{defn}[Hecke pair] \label{defn:Hecke_pair}
A pair of stacks $\mathcal{Y}$ and $\mathcal{Z}$ (see \cref{subsubsec:prelim_stacks} for our conventions regarding stacks) equipped with a morphism $\mathcal{Y} \to \mathcal{Z}$ is said to be a \emph{Hecke pair} if the following conditions are satisfied:
\begin{myenum}{(\roman*)}
	\item $\mathcal{Y}$ and $\mathcal{Z}$ are perfect and locally of finite presentation;
	\item for any finite CW complex $M$ of dimension at most 2, $\mathcal{Y}^M$ and $\mathcal{Z}^M$ are bounded; and
	\item for any open embedding of $2$-dimensional manifolds $N \to M$, $(\mathcal{Y}, \mathcal{Z})^{\lbar{M} \setminus N, \lbar{M}}$ is bounded.
\end{myenum}
\end{defn}

By \cref{rmk:perfectness_lfp_mapping_stacks}, all stacks that appear in the proof of \cref{thm:main} are already perfect and locally of finite presentation. The last two conditions of \cref{defn:Hecke_pair} guarantee that these stacks are also bounded.

\subsubsection{$BP\to BG$ is a Hecke pair} 
The main case of interest to us indeed satisfies this condition.
\begin{lem}
For any homomorphism of affine algebraic group $H\to G$, $BH\to BG$ is a Hecke pair.
\end{lem}
\begin{proof}
It's clear that $BH$ and $BG$ are locally of finite presentation. Moreover, they are perfect, by~\cite{ben-zvi_integral_2010}.

\cref{lem:BC_cMap_qsmooth_CW_dim2,lem:BHBG_cMap_qsmooth} below show that the two last conditions of \cref{defn:Hecke_pair} are also satisfied and the proof is completed.
\end{proof}

\begin{lem} \label{lem:BG_cMap_smooth_CW_dim1}
For any affine algebraic group $G$ and any finite CW complex $M$ of dimension at most $1$, $BG^{M}$ is smooth.
\end{lem}
\begin{proof}
Note that any such $M$ is homotopy equivalent to a finite disjoint union of points and wedges of circles. Thus, $BG^M$ is a finite product of stacks of the forms $BG$ and $\frac{G}{G} \times_{BG} \cdots \times_{BG} \frac{G}{G}$ where $\frac{G}{G}$ is the stack quotient of $G$ by itself via the conjugation action. These are smooth and hence, we are done.
\end{proof}

\begin{lem} \label{lem:BC_cMap_qsmooth_CW_dim2}
For any affine algebraic group $G$ and any finite CW complex $M$ of dimension at most $2$, $BG^M$ is quasi-smooth in the sense of~\cite{arinkin_singular_2015}. In particular, $BG^M$ is bounded.
\end{lem}
\begin{proof}
We prove this inductively based on the CW presentation of $M$. When $M$ is at most $1$-dimensional, this is already done in the previous lemma. Now, $M$ is built up inductively from pushout diagrams of the following form
\[
\begin{tikzcd}
	S^1 \ar{d} \ar{r} & M' \ar{d} \\
	D^2 \ar{r} & M
\end{tikzcd}
\]
This gives the following pullback square
\[
\begin{tikzcd}
	BG^M \ar{r} \ar{d}{f} & BG \ar{d}{g} \\
	BG^{M'} \ar{r} & BG^{S^1}
\end{tikzcd}
\]
Since $BG$ and $BG^{S^1} \simeq \frac{G}{G}$ are smooth, $g$ is a quasi-smooth map. Thus, so is $f$. By inductive hypothesis, $BG^{M'}$ is quasi-smooth. Thus, so is $BG^M$. 
\end{proof}

\begin{lem} \label{lem:BHBG_cMap_qsmooth}
For any homomorphism of affine algebraic groups $H \to G$ and any open embedding of $2$-dimensional manifolds $N\to M$, $(BH, BG)^{\lbar{M}\setminus N, \lbar{M}}$ is quasi-smooth, and hence, bounded.
\end{lem}
\begin{proof}
By definition, we have the following pullback square
\[
\begin{tikzcd}
	(BH, BG)^{\lbar{M}\setminus N, \lbar{M}} \ar{r} \ar{d}{f} & BG^{\lbar{M}} \ar{d}{g} \\
	BH^{\lbar{M}\setminus N} \ar{r} & BG^{\lbar{M}\setminus N}
\end{tikzcd}
\]

Without loss of generality, we can assume that $M$ (and hence, also $\lbar{M}$) is connected. Then, $\lbar{M} \setminus N$ has the homotopy type of a CW complex of dimension at most $1$. By \cref{lem:BG_cMap_smooth_CW_dim1}, $BG^{\lbar{M}\setminus N}$ and $BH^{\lbar{M}\setminus N}$ are smooth and by \cref{lem:BC_cMap_qsmooth_CW_dim2} $BG^{\lbar{M}}$ is quasi-smooth. Thus, $g$ is quasi-smooth. Hence, so is $f$. But then, this implies that $(BH, BG)^{\lbar{M} \setminus N, \lbar{M}}$ is also quasi-smooth and we are done.
\end{proof}

\subsubsection{A nil-isomorphism}
The pair $BP \to BG$ in fact has another simplifying property, namely, the natural map $BP \to (BP, BG)^{S^1, D^2}$ is a nil-isomorphism, i.e. the corresponding morphism between de Rham prestacks is an isomorphism. Indeed,
\[
	(BP, BG)^{S^1, D^2} \simeq {\textstyle\frac{P}{P}} \times_{\frac{G}{G}} BG \simeq {\textstyle{\frac{P}{P}}} \times_{\frac{G}{P}} BP \simeq \mathfrak{n}^*_P[-1]/P, \teq\label{eq:computation_of_H_GP}
\]
where $\mathfrak{n}_P^*$ is the linear dual of the nilpotent radical of the Lie algebra of $P$. The underlying de Rham prestack of this is simply $BP_{\dR}$.

By~\cref{subsubsec:IndCoh_0_special_niliso}, we see that
\begin{align*}
	\Hecke_2(BP, BG) 
	&\defeq \IndCoh_0\left(\left((BP, BG)^{S^1, D^2}\right)^\wedge_{BP}\right) \\
	&\simeq \IndCoh((BP, BG)^{S^1, D^2}) \\
	&\simeq \IndCoh(\LocSys_{G, P}(D^2, S^1)) \teq\label{eq:E_2_Hecke_just_IndCoh}
\end{align*}
which is precisely $\Hecke_{G, P}$ appearing in the introduction,~\cref{subsubsec:intro:Hecke_cat}. Note that the important point is that at the local level of a disk, $\IndCoh_0$ does not make an appearance! However, $\IndCoh_0$ appears naturally after taking factorization homology.

From the discussion above, we thus obtain the following corollary of \cref{thm:main}.

\begin{cor} \label{cor:main_Eis}
For any topological surface $M$ (with possibly non-empty boundary), we have
\[
	\int_M \Hecke_{G, P} \simeq \IndCoh_0\left(\left((BP, BG)^{\partial M, M}\right)^{\wedge}_{BP^M}\right) \simeq \IndCoh_0\left(\LocSys_{G, P}(M, \partial M)^\wedge_{\LocSys_P(M)}\right) \simeq \Eis_{G, P}(M).
\]
\end{cor}

\subsection{Eisenstein series on a non-compact surface} \label{subsec:non-compact_case}
We see in~\cref{eq:E_2_Hecke_just_IndCoh} that
\[
	\Eis_{G, P}(D^2) = \Hecke_{G, P} \simeq \IndCoh(\LocSys_{G, P}(D^2, S^1)).
\]
In particular, it says that Eisenstein series for a $2$-dimensional disk only involve $\IndCoh$ rather than the more complicated $\IndCoh_0$. In this subsection, we show a similar statement for non-compact topological surfaces. More precisely, for a non-compact topological surface $M$, we will show that $\Eis_{G, P}(M)$ is naturally a full-subcategory of $\IndCoh(\LocSys_{G, P})(M, \partial M)$. The key point is given by the following result.

\begin{prop} \label{prop:closed_embedding_non-compact_M}
Let $G$ be an affine algebraic group, $H$ a closed subgroup, and $M$ a non-compact surface. Then, the natural map
\[
	BH^M \simeq \LocSys_H(M) \to \LocSys_{G, H}(M, \partial M) \simeq BH^{\partial M} \times_{BG^{\partial M}} BG^M
\]
is a closed embedding of stacks. 
\end{prop}
\begin{proof}
The proof can be best visualized using \cref{fig:punctured_surface} where the case of the thrice-punctured surface of genus $4$ is illustrated. To start, note that any punctured surface is homotopy equivalent to something of same form as the bottom right of \cref{fig:punctured_surface}.\footnote{Note that by~\cref{subsubsec:abuse_notation_boundary}, we are really thinking about the associated compact surface with boundary.} The important point is that the boundary is a string of circles.

\begin{figure}[htb]
	\centering
	\includegraphics[height=2.6in]{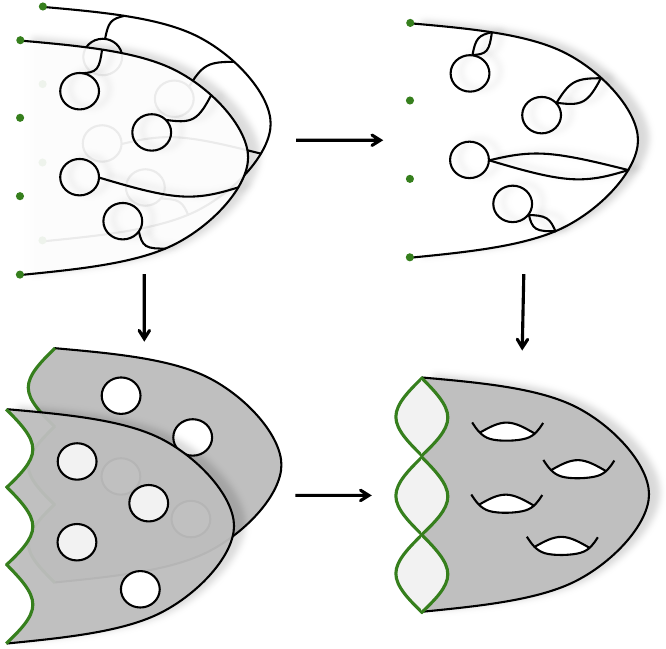}
	\caption{Gluing a punctured surface.} \label{fig:punctured_surface}
	\begin{descriptionBox}
		This figure illustrates the case of a thrice-punctured surface of genus $4$. The given diagram is a pushout square; vertical maps are injective and horizontal ones are quotient maps. Similarly to \cref{fig:pair_of_pants}, these pictures illustrate objects in $\Spc^{\Delta^1}$ where the green parts denote the first factor. For example, for $(N', N) \in \Spc^{\Delta^1}$ represented by any one of the four objects in the square, $\LocSys_{G, H}(N', N)$ is the moduli space of $G$-local system on the whole of $N$ plus a $H$-reduction on the green parts given by $N'$. 
		\end{descriptionBox}
\end{figure}

Now, $(\partial M, M) \in \Spc^{\Delta^1}$ can be sliced into two ``sheets.'' We denote the resulting object by $(\wtilde{\partial} F, F) \in \Spc^{\Delta^1}$, which is illustrated by the bottom left of \cref{fig:punctured_surface}. Here, we use $\wtilde{\partial} F$ rather than $\partial F$ to emphasize that it is \emph{not} the boundary $F$ but rather, it comes from the boundary of $M$. Note that $F$ stands for faces.

Let $(\wtilde{\partial} S, S)$ and $(\wtilde{\partial} Q, Q)$ be elements in $\Spc^{\Delta^1}$ represented by the top left and right of \cref{fig:punctured_surface} respectively, where $S$ and $Q$ stand for skeleton and quotient respectively.\footnote{Glue would have been better but we already use $G$ to denote the group $G$.} It is clear that we have a pushout diagram (which is the one illustrated by \cref{fig:punctured_surface})
\[
\begin{tikzcd}
	(\wtilde{\partial} S, S) \ar{d} \ar{r} & (\wtilde{\partial} Q, Q) \ar{d} \\
	(\wtilde{\partial} F, F) \ar{r} & (\partial M, M)
\end{tikzcd} \teq\label{eq:push_out_punctured_surface}
\]

Applying $(BH, BG)^{-, -}$ (see~\cref{subsubsec:cMap_defn}) to~\cref{eq:push_out_punctured_surface} and $BH^-$ to the second terms of~\cref{eq:push_out_punctured_surface}, we obtain the following Cartesian squares
\[
\begin{tikzcd}
	(BH, BG)^{\wtilde{\partial} S, S} & (BH, BG)^{\wtilde{\partial} Q, Q} \ar{l} \\
	(BH, BG)^{\wtilde{\partial} F, F} \ar{u} & (BH, BG)^{\partial M, M} \ar{u} \ar{l}
\end{tikzcd} \qquad
\begin{tikzcd}
	BH^{S} & BH^{Q} \ar{l} \\
	BH^{F} \ar{u} & BH^{M} \ar{u} \ar{l}
\end{tikzcd} \teq\label{eq:pullback_punctured_surface_stacks}
\]
Note that the bottom right terms of the two squares are $\LocSys_{G, H}(\partial M, M)$ and $\LocSys_H(M)$ respectively. Now, to show that the natural map $\LocSys_H(M) \to \LocSys_{G, H}(\partial M, M)$ is a closed embedding, it suffices to show that the natural map from each of the three other terms of the square on the right to the corresponding term of the square on the left is a closed embedding.

We will now prove this for each of the maps. In what follows, we will use fractions to denote stack quotients with respect to conjugation actions. Moreover, let $g$ and $k$ denote the genus of $M$ and the number of punctures, respectively. 
\begin{myenum}{--}
	\item For $F$ (bottom left of the squares in~\cref{eq:pullback_punctured_surface_stacks}), note that $(\wtilde{\partial} F, F)$ is homotopy equivalent to $(\pt, (S^1)^{\vee g})^{\sqcup 2}$ where $(S^1)^{\vee g}$ is a wedge of $g$ circles. Thus, $BH^F \simeq (\frac{H^g}{H})^2$ and moreover,
	\[
		(BH, BG)^{\wtilde{\partial} F, F} \simeq \left(BH \times_{BG} \frac{G^g}{G}\right)^2 \simeq \left(\frac{G^g}{H}\right)^2,
	\]
	which clearly receives a closed embedding from $(\frac{H^g}{H})^2$.
	
	\item For $S$ (top left of the squares in~\cref{eq:pullback_punctured_surface_stacks}), note that $(\wtilde{\partial} S, S)$ is homotopy equivalent to $((\pt^{\sqcup k-1}, \pt^{\sqcup k-1}) \sqcup (\pt\sqcup \pt, (S^1)^{\vee g}))^{\sqcup 2}$ where the second part is a wedge of $g$ circles together with $2$ marked points. Thus, $BH^{S} \simeq (BH^{k-1} \times \frac{H^g}{H})^2$ and moreover
	\[
		(BH, BG)^{\wtilde{\partial} S, S} \simeq \left(BH^{k-1} \times (BH\times BH) \times_{BG\times BG} \frac{G^g}{G}\right)^2.
	\]
	It suffices to show that the natural map
	\[
		\frac{H^g}{H} \to (BH\times BH) \times_{BG\times BG} \frac{G^g}{G}
	\]
	is a closed embedding. Note that the RHS is equivalent to
	\[
		(BH\times BH) \times_{BG \times BG} BG \times_{BG} \frac{G^g}{G} \simeq BH\times_{BG} BH \times_{BG} \frac{G^g}{G} \simeq BH\times_{BG} \frac{G^g}{H}.
	\]
	Since $\frac{H^g}{H} \to \frac{G^g}{H}$ is a closed embedding, so is $BH \times_{BG} \frac{H^g}{H} \to BH \times_{BG} \frac{G^g}{H}$. But now, the graph $\frac{H^g}{H} \to BH\times_{BG} \frac{H^g}{H}$ is a closed embedding since $BH$ is separated (in fact, even proper when $H$ is chosen to be a parabolic subgroup) over $BG$. We thus obtain that $\frac{H^g}{H} \to BH\times_{BG} \frac{G^g}{H}$ is a closed embedding, completing the case of $S$.

	\item For the case of $Q$ (top right of the squares in~\cref{eq:pullback_punctured_surface_stacks}), we obtain the desired result by arguing similarly to the case of $S$.
\end{myenum}
\end{proof}	

\begin{thm} \label{thm:Eis_non-compact}
Let $G$ be an affine algebraic group, $H$ a closed subgroup, and $M$ a non-compact topological surface. Then
\[
	\Eis_2(BH, BG)(M) \simeq \IndCoh(\LocSys_{G, H}(M, \partial M)^\wedge_{\LocSys_{H}(M)}) \xhookrightarrow{\text{f.f.}} \IndCoh(\LocSys_{G, H}(M, \partial M)),
\]
where f.f. stands for fully faithful. In particular, when $H=P$ is a parabolic subgroup of $G$, we have
\[
	\Eis_{G, P}(M) \simeq \IndCoh(\LocSys_{G, P}(M, \partial M)^\wedge_{\LocSys_P(M)}) \xhookrightarrow{\text{f.f.}} \IndCoh(\LocSys_{G, P}(M, \partial M)).
\]
\end{thm}
\begin{proof}
Since $M$ has the homotopy type of a $1$-dimensional CW complex, $BH^M = \LocSys_H(M)$ is smooth, by \cref{lem:BG_cMap_smooth_CW_dim1}. Thus,
\[
	\Eis_2(BH, BG)(M) \defeq \IndCoh_0(\LocSys_{G, H}(M, \partial M)^\wedge_{\LocSys_H(M)}) \simeq \IndCoh(\LocSys_{G, H}(M, \partial M)^\wedge_{\LocSys_H(M)}).
\]
By \cref{prop:closed_embedding_non-compact_M}, $\LocSys_H(M) \to \LocSys_{G, H}(M, \partial M)$ is a closed embedding. Thus, we see that $\IndCoh(\LocSys_{G, H}(M, \partial M)^\wedge_{\LocSys_H(M)})$ is the full subcategory of $\IndCoh(\LocSys_{G, H}(M, \partial M))$ consisting precisely of ind-coherent sheaves whose set-theoretic support is $\LocSys_H(M)$.
\end{proof}

\subsubsection{An example}
We now consider the extreme case where $H = \{1\} \subset G$ is the trivial subgroup of $G$. Recall that
\[
	\LocSys_{G, \{1\}}(D^2, S^1) \simeq \pt \times_{\frac{G}{G}} BG \simeq \pt \times_{G} \pt \simeq \mathfrak{g}[-1] \simeq \Spec \Sym(\mathfrak{g}^*[1]).	
\]
Thus, $\LocSys_{G, \{1\}}(D^2, S^1)^{\wedge}_{\triv} \simeq \LocSys_{G, \{1\}}(D^2, S^1)$ and hence,
\[
	\Hecke_2(\pt, BG) \simeq \IndCoh(\LocSys_{G, \{1\}}(D^2, S^1)^{\wedge}_{\triv}) \simeq \IndCoh(\mathfrak{g}[-1]).
\]

Let $M$ be any non-compact surface. Then,
\begin{align*}
\int_M \IndCoh(\mathfrak{g}[-1]) 
&\simeq \Eis_2(\pt, BG)(M) \\
&\simeq \IndCoh(\LocSys_{G, \{1\}}(M, \partial M)^\wedge_{\LocSys_{\{1\}} M}) \\
&\simeq \IndCoh(\LocSys_{G, \{1\}}(M, \partial M)^\wedge_{\triv}),	
\end{align*}
where $\triv$ denotes the trivial local system. This category is the full subcategory of
\[
	\IndCoh(\LocSys_{G, \{1\}}(M, \partial M))
\]
consisting of all ind-coherent sheaves supported at the trivial local system. 

\section*{Acknowledgments}
The authors thank D. Beraldo, J. Campbell, L. Chen, G. Dhillon, J. Francis, D. Gaitsgory, D. Nadler, P. Shan, G. Stefanich, and P. Yoo for stimulating conversations and email exchanges regarding the subject matter of the paper. We thank the anonymous referee for many helpful comments and suggestions.

The paper was written when Q. Ho was a postdoc in Hausel group at IST Austria, supported by the Lise Meitner fellowship, Austrian Science Fund (FWF): M 2751. P. Li is partially supported by the National Natural Science Foundation of China (Grant No. 12101348).

\bibliography{eis_tqft}
\end{document}